\documentclass{article}

\usepackage[utf8]{inputenc} 
\usepackage[T1]{fontenc}    
\usepackage{hyperref}       
\usepackage{url}            
\usepackage{booktabs}       
\usepackage{amsfonts}       
\usepackage{nicefrac}       
\usepackage{microtype}      

\usepackage{amsfonts, amsmath, amssymb, amsthm}
\usepackage{dsfont}
\usepackage[top=1in,bottom=1in,left=1in,right=1in]{geometry}
\usepackage{mathrsfs}
\usepackage{verbatim}
\usepackage{graphicx}
\usepackage{authblk}
\usepackage{url}
\usepackage{blindtext}
\usepackage{color, colortbl}
\usepackage[table]{xcolor}
\usepackage{algorithm}
\usepackage{algorithmic}
\usepackage{enumitem}

\DeclareMathOperator*{\argmax}{arg\, max}
\DeclareMathOperator*{\argmin}{arg\, min}

\newtheorem{thm}{Theorem}

\newtheorem{lem}{Lemma}
\newtheorem{cor}{Corollary}

\newtheorem{remark}{Remark}
\newtheorem{definition}{Definition}
\newtheorem{hyp}{Assumption}
\newenvironment{skproof}{%
  \proof}{\endproof}

\author{%
  Solenne Gaucher \\
  Laboratoire de Mathématiques d'Orsay\\
  Université Paris-Saclay,\\
  91405, Orsay, France \\
  \texttt{solenne.gaucher@math.u-psud.fr}
}
\title{Finite Continuum-Armed Bandits}

\begin{document}

\maketitle

\begin{abstract}
We consider a situation where an agent has $T$ ressources to be allocated to a larger number $N$ of actions. Each action can be completed at most once and results in a stochastic reward with unknown mean. The goal of the agent is to maximize her cumulative reward. Non trivial strategies are possible when side information on the actions is available, for example in the form of covariates. Focusing on a nonparametric setting, where the mean reward is an unknown function of a one-dimensional covariate, we propose an optimal strategy for this problem. Under natural assumptions on the reward function, we prove that the optimal regret scales as $O(T^{1/3})$ up to poly-logarithmic factors when the budget $T$ is proportional to the number of actions $N$. When $T$ becomes small compared to $N$, a smooth transition occurs.  When the ratio $T/N$ decreases from a constant to $N^{-1/3}$, the regret increases progressively up to the $O(T^{1/2})$ rate encountered in continuum-armed bandits.

\end{abstract}

\section{Introduction}
\subsection{Motivations}

Stochastic multi-armed bandits have been extensively used to model online decision problems under uncertainty : at each time step, an agent must choose an action from a finite set, and receives a reward drawn i.i.d. from a distribution depending on the action she has selected. By choosing the same action over and over again, she can learn the distribution of the rewards for performing this action. The agent then faces a trade-off between collecting information on the mechanism generating the rewards, and taking the best action with regards to the information collected, so as to maximise her immediate reward.

In some real-life situations, the agent can complete each action at most once, and does not have enough resource to complete all of them. Her decisions can be rephrased in terms of allocating limited resources between many candidates. The agent cannot estimate the reward of an action by performing it several times, and must rely on additional information to construct her strategy. In many situations, covariates providing information on the actions are available to the agent. Then, the expected reward for taking an action can be modelled as a (regular) function of the corresponding covariate. Thus, similar actions give rise to similar rewards. This problem is motivated by the following examples.
\vspace{-0.1 cm}
\begin{itemize}[leftmargin=*]
\item \textbf{Allocation of scarce resources.}
The response of an individual to medical treatment can be inferred from contextual information describing this patient. When this treatment is expensive or short in supply, decision-makers aim at efficiently selecting recipients who will be treated, so as to maximise the number of beneficial interventions \cite{Osteoporose}. During epidemic crises, lack of medical resources may force hospital staff to progressively identify patients that are more likely to recover based on indicators of their general health status, and prioritize them in the resource allocation. Similar questions arise when determining college admission so as to optimize the number of successful students \cite{Fairness}, or allocating financial aid to individuals most likely to benefit from it.

 \item \textbf{Contextual advertisement with budget.} A common form of payment used in online advertisement is pay-per-impression: the advertiser pays a fixed fee each time an ad is displayed \cite{BanditBudgetCombes}, and the budget of an advertising campaign determines the number of users who view the advertisement. It has been shown in \cite{Exposure} that click-through rates decrease steeply as users are exposed over and over to the same recommendation. Advertisers may therefore prefer to display their campaign to a new potential customer rather than to an already jaded one, so that each user will view the campaign at most once. Those users are often described by features including demographic information as well as previous online activities. Advertisers want to leverage this contextual information so as to focus on users that are more likely to click on the ad banner.

\item \textbf{Pair matching.}  Finding good matches between pairs of individuals is an ubiquitous problem. Each pair of individuals represents an action : the agent sequentially selects $T$ pairs, and receives a reward each time the pair selected corresponds to a good matching \cite{giraud2019pair}. In many settings, the agent has access to information describing either individuals or pairs of individuals. For example, online gaming sites may want to pair up players of similar level or complementary strength; dating applications may use information provided by the users to help them find a partner. Similarly, biologists studying  protein-protein interaction networks will sequentially test pairs of proteins to discover possible interactions. Such experiments are however costly, difficult and time-consuming, and leveraging information describing those proteins can help researchers focus on pairs more likely to interact \cite{PPI}. 
\end{itemize}

In these settings, the decision maker can complete each action (i.e., select each internet user, patient, college candidate or pair of individuals) at most once; however by selecting an action, she learns about the expected rewards of similar actions. We model the dependence of the expected reward on the variable describing this action in a non-parametric fashion, and rephrase our problem by using terminology from the bandit literature.

\textbf{The Finite Continuum-Armed Bandit (F-CAB) problem : }
An agent is presented with a set of $N$ arms described by covariates $\{a_1, a_2, ..., a_N\}$ in a continuous space $\mathcal{X}$ (the arm $i$ will henceforth be identified with its covariate $a_i$). The agent is given a budget $T$ to spend on those arms, where $T$ is typically a fraction $p$ of the number of available arms $N$. At each step  $t \leq T$, the agent pulls an arm $\phi(t)$ among the arms that have not been pulled yet, and receives the corresponding reward $y_{\phi(t)} \in [0,1]$. Conditionally on $\{a_1, a_2, ..., a_N\}$, the rewards $y_{i}$ are sampled independently from some distribution with mean  $m(a_{i})$, where $m : \mathcal{X} \rightarrow [0,1]$ is the (unknown) mean reward function. The aim of the agent is to maximise the sum of the rewards she receives.

The F-CAB problem is closely related to the classical continuum-armed bandit problem. This problem, first introduced in \cite{KleinbergCont}, extends multi-armed bandits to continuous sets of actions.  At each step, the agent takes an action indexed by a point of her choosing in a continuous space $\mathcal{X}$. In order to maximise her gains, she must explore the space $\mathcal{X}$ so as to find and exploit one of the maximas of the mean reward function. The assumption that the agent can choose arbitrarily any action corresponding to any covariate, unrealistic in many real-life situations, is relaxed in the F-CAB model. Moreover in the F-CAB setting, the agent can pull each arm at most once. Thus she must endeavour to find and exploit a large set of good arms, as she cannot focus on a single arm corresponding to a maxima.

\subsection{Related work}

To the best of the author's knowledge, continuum-armed bandits without replacement have not been considered before. On the other hand, variants to the multi-armed bandit problem were proposed to relax the assumption that the agent can choose any action an infinite number of time.

In \cite{MortalMAB}, the authors consider a multi-armed bandit problem with infinitely many arms, whose rewards are drawn i.i.d. from some known distribution. Each arm can only be pulled a finite number of times before it dies. Algorithms developed for this problem heavily rely on the knowledge of the distribution of the arms, and on the fact that an infinite number of good arms is always available to the player, both assumptions that are violated in our setting.

Closer to our problem is \cite{scratch_game} : the authors study the problem of scratch game, where each arm can be pulled a limited number of time before dying. They bound the weak regret, defined as the difference between $T \times m_{(\phi^*(T))}$ and the cumulative reward of the player, where $m_{(\phi^*(T))}$ is the expected reward of the $T$-th armed pulled by an oracle strategy. Since the reward of the arm pulled by this oracle strategy decreases at each step, its cumulative reward can be much larger than  $T \times  m_{(\phi^*(T))}$ (both can differ by a linear factor). Thus, the weak regret can be significantly lower than the classical regret, which we control in this paper.

Another related problem is that of budgeted bandits with different budgets for each arm : the decision maker faces a multi-armed bandit problem with constraints on the number of pull of each arm. This problem is studied in \cite{Exposure}: the authors assume that the number of arms is fixed, and that the budget of each arm increases proportionally to the number of steps $T$. They provide numerical simulations as well as asymptotic theoretical bounds on the regret of their algorithm. More precisely, they show that in the limit $T \rightarrow \infty$, all optimal arms but one have died before time $T$ : thus, when the budget of each arm and the total number of pulls $T$ are sufficiently large, the problem reduces to a classical multi-armed bandit. By contrast, in the F-CAB setting we can pull each arm at most once and do not attain this regime. Our technics of proof require therefore more involved, non-asymptotic regret bounds.

\subsection{Contribution and outline}
In this paper, we present a new model for finite continuum-armed bandit motivated by real-world applications. In this resource allocation problem, each action is described by a continuous covariate, and can be taken at most once. After some preliminary discussions, we restrict our attention to one-dimensionnal covariates and introduce further assumptions on the distribution of the covariates $a_i$ and on the mean payoff function $m$ in Section \ref{sec:setup_prel}. In Section \ref{sec:ucbf}, we present an algorithm for this problem, and establish a non-asymptotic upper-bound on the regret of this algorithm. More precisely, we prove that when the budget $T$ is a fixed proportion of the number of arms, with high probability, $R_T = O(T^{1/3}\log(T)^{4/3})$. This rate is faster than all regret rates achievable in the classical continuum armed bandit under similar assumptions on the mean reward function. Indeed, the authors of  \cite{Auer_continuum} show that regret for the classical continuum-armed bandits problem is typically of order $O(T^{1/2}\log(T))$. On the other hand, we show that when the budget $T$ becomes small compared to the number of arms $N$, the regret rate smoothly increases. In the limit where the ratio $T/N$ decreases to $N^{-1/3}\log(N)^{2/3}$, the regret increases progressively up to the $O(T^{1/2}\log(T))$ rate encountered in classical continuum-armed bandit problems. Moreover, we derive matching lower bounds on the regret, showing that our rate is sharp up to a poly-logarithmic factor. Extensions of our methods to multi-dimensional covariates are discussed in Section \ref{sec:discussion} and detailed in the Appendix. We provide high level ideas behind those results throughout the paper but defer all proofs to the Appendix.

\section{Problem set-up}
\label{sec:setup_prel}
\subsection{Preliminary discussion}

In the F-CAB problem, each arm can be pulled at most once, and exploration is made possible by the existence of covariates describing the arms. This framework is related to the classical Continuum-Armed Bandit problem, which we recall here.

\textbf{The Continuum-Armed Bandit (CAB) problem: }
At each step $t$, an agent selects any covariate $a_{t} \in \mathcal{X}$, pulls an arm indexed by this covariate and receives the corresponding reward $y_{t} \in [0,1]$. Here again, the rewards for pulling an arm $a\in [0,1]$ are drawn i.i.d. conditionally on $a$ from some distribution with mean $m(a)$. The agent aims at maximising her cumulative reward.

By contrast to the CAB setting, where the agent is free to choose any covariate in $\mathcal{X}$, in the F-CAB setting she must restrict her choice to the ever diminishing set of available arms. The usual trade-off between exploration and exploitation breaks down, as  the agent can pull but a finite number of arms in any region considered as optimal. Once those arms have been pulled, all effort spent on identifying this optimal region may become useless. On the contrary, in the CAB setting the agent may pull arms in a region identified as optimal indefinitely. For this reason, strategies lead to lower cumulative reward in the F-CAB setting than that in the less constrained CAB setting.

Nonetheless, this does not imply that F-CAB problems are more difficult than CAB ones in terms of regret. The difficulty of a problem is often defined, in a minimax sense, as the performance of the best algorithm on a worst problem instance. In bandit problems, the performance of a strategy $\phi$ is often characterised as the difference between its expected cumulative reward, and that of an agent knowing in hindsight the expected rewards of the different arms. At each step $t = 1, ..., T$, this oracle agent pulls greedily the arm $\phi^*(t)$, where $\phi^*$ denote a permutation of $\{1, ..., N\}$ such that $m(a_{\phi^*(1)}) \geq m(a_{\phi^*(2)}) \geq ... \geq m(a_{\phi^*(N)})$. Note that this agent receives an expected cumulative reward of $\sum_{t\leq T} m(a_{\phi^*(t)})$, which is lower than $T \times \max_a m(a)$. Thus the regret, defined as the difference between the cumulative reward of $\phi
^*$ and that of our strategy, is given by

\begin{equation*}
     R_T = \underset{1 \leq t \leq T}{\sum} m(a_{\phi^*(t)}) - \underset{1 \leq t \leq T}{\sum} m(a_{\phi(t)}).
\end{equation*}

The difficulty of the F-CAB problem is governed by the ratio $p = T/N$. In the limit $p \rightarrow 1$, the problem becomes trivial as any strategy must pull all arms, and all cumulative rewards are equal. Opposite to this case, in the limit $p \to 0$,  choosing $a_{\phi(t)}$ from the large set of remaining arms becomes less and less restrictive, and we expect the problem to become more and more similar to a CAB. To highlight this phenomenon, we derive upper and lower bounds on the regret that explicitly depend on $p$. We show that when $p \in (0,1)$ is a fixed constant, i.e. when the budget is proportional to the number of arms, lower regret rates can be achieved for the F-CAB problem than for the CAB problem. To the best of the author's knowledge, it is the first time that this somewhat counter-intuitive phenomenon is observed; however it is consistent with previous observations on rotting bandits \cite{Rotting_bandits}, in which the expected reward for pulling an arm decreases every time this arm is selected. Like in the F-CAB model, in rotting bandits the oracle agent receives ever decreasing rewards. The authors of \cite{Rotting_bandits_no_harder} show that this problem is no harder than the classical multi-armed bandit : although the cumulative rewards are lower than those in the classical multi-armed bandit setting, it does not imply that strategies should suffer greater regrets. This phenomenon is all the more striking in the F-CAB setting, as we show that strategies can in fact achieve lower regrets. Finally, we verify that when $p \rightarrow 0$, the regret rate increases. In the limit where $p = N^{-1/3}\log(N)^{2/3}$, the problem becomes similar to a CAB and the regret rate increases up to the rate encountered in this setting.

\subsection{Assumptions on the covariates and the rewards}

 While in general the covariates $a_i$ could be multivariate, we restrict our attention to the one-dimensional case, and assume that $\mathcal{X} = [0,1]$. The multivariate case is discussed and analysed in Section \ref{sec:discussion} and in the Appendix. Focusing on the one-dimensional case allows us to highlight the main novelties of this problem by avoiding cumbersome details. We make the following assumption on the distribution of the arms. 

\begin{hyp}\label{hyp:uniforme}
For $i = 1, ..., N$, $a_i \overset{i.i.d.}{\sim} \mathcal{U}([0,1])$.
\end{hyp}

By contrast to the CAB setting, where one aims at finding and pulling arms with rewards close to the maxima of $m$, in a F-CAB setting the agent aims at finding and pulling the $T$ best arms : the difficulty of the problem thus depends on the behaviour of $m$ around the reward of the $T$-th best arm $m(a_{\phi^*(T)})$. Under Assumption \ref{hyp:uniforme}, we note that $\mathbb{E}[m(a_{\phi^*(T)})] = M$, where $M$ is defined as 
\[M = \min\left\{A : \lambda\left(\{x : m(x)\geq A\}\right) < p \right\}\]
and $\lambda$ is the Lebesgue measure. In words, we aim at identifying and exploiting arms with expected rewards above the threshold $M$. We therefore say that an arm $a_i$ is optimal if $m(a_i) \geq M$, and that it is otherwise sub-optimal. Moreover, we say that an arm $a_i$ is sub-optimal (respectively optimal) by a gap $\Delta$ if $0 \leq M - m(a_i) \leq \Delta$ (respectively $0 \leq m(a_i) - M \leq \Delta$).

We make the following assumptions on the mean reward function. First, note that if $m$ varies sharply, the problem becomes much more difficult as we cannot infer the value of $m$ at a point based on rewards obtained from neighbouring arms. In fact, if $m$ presents sharp peaks located at the $T$ optimal arms, any reasonable strategy must suffer a linear regret. In order to control the fluctuations of $m$, we assume that it is weakly Lipschitz continuous around the threshold $M$.

\begin{hyp}[Weak Lipschitz condition]\label{hyp:lip}
There exists $L >0$ such that, for all $(x,y) \in [0,1]^2$,
\begin{equation}
\vert m(x) - m(y) \vert \leq \max\{\vert M - m(x)\vert, L\left\vert x - y\right \vert\}.\label{eq:Lipshitz}
\end{equation}
\end{hyp}

Assumption \ref{hyp:lip} is closely related to Assumption A2 in \cite{x-armed}. It requires that the mean reward function $m$ is $L$-Lipschitz at any point $x'$ such that $m(x') = M$ : indeed, in this case the condition states that for any $y$, $\vert m(x) - m(y) \vert \leq L\left\vert x - y\right \vert$. On the other hand, $m$ may fluctuate more strongly around any point $x$ whose expected reward is far from the threshold $M$.

Bandit problems become more difficult when many arms are slightly sub-optimal. Similarly, the F-CAB problem becomes more difficult if there are many arms with rewards slightly above or under the threshold $M$, since it is hard to classify those arms respectively as optimal and sub-optimal. This difficulty is captured by the measure of points with expected rewards close to $M$. 

\begin{hyp}[Margin condition]\label{hyp:beta}
There exists $Q>0$ such that for all $\epsilon \in (0,1)$,
\begin{equation}
\lambda\left(\left\{ x :  \left\vert M - m(x)\right\vert \leq  \epsilon \right\}\right) \leq Q \epsilon.\label{eq:margin}
\end{equation}
\end{hyp}

In the classical CAB setting, lower bounds on the regret are of the order $O(T^{1/2})$ under similar margin assumptions, and they become $O(T^{2/3})$ when these margin assumptions are not satisfied. In the F-CAB, Assumption \ref{hyp:beta} allow us to improve regret bounds up to $O(T^{1/3}p^{-1/3})$.  It is altogether not too restrictive, as it is verified if $m$ has finitely many points $x$ such that $m(x) = M$, and has non vanishing first derivatives at those points. Note that if the margin assumption and the weak Lipschitz assumption hold simultaneously for some $L,Q>0$, we must have $QL \geq 1$.

\section{UCBF : Upper Confidence Bound algorithm for Finite continuum-armed bandits}

\label{sec:ucbf}

\subsection{Algorithm}
\label{subsec:ucbf}

We now describe our strategy, the Upper Confidence Bound for Finite continuum-armed bandits (UCBF). It is inspired from the algorithm UCBC introduced in \cite{Auer_continuum} for CAB. 

\begin{algorithm}
\caption{Upper Confidence Bound for Finite continuum-armed bandits (UCBF)}
\begin{algorithmic}
\STATE \textbf{Parameters:} $K, \delta$
\STATE \textbf{Initialisation:} Divide $[0,1]$ into $K$ intervals $I_k$ with $I_k = [\frac{k-1}{K}, \frac{k}{K})$for $k \in \{1, ..., K-1\}$ and $I_K = [\frac{K-1}{K},1]$. Let $N_k = \sum_{1 \leq i \leq N} \mathds{1}\{a_i \in I_k\}$ be the number of arms in the interval $I_k$. Define the set of intervals alive as the set of intervals $I_k$ such that $N_k \geq 2$. Pull an arm uniformly at random in each interval alive.
\FOR{$t = K+1, ..., T$} 
    \STATE $-$ Select an interval $I_k$ that maximizes $\widehat{m}_{k}(n_k(t-1)) + \sqrt{\frac{\log(T/\delta)}{2n_k(t-1)}}$ among the set of alive intervals, where $n_k(t-1)$ is the number of arms pulled from $I_k$ by the algorithm before time $t$, and $\widehat{m}_{k}(n_k(t-1))$ is the average reward obtained from those $n_k(t-1)$ samples.
    \STATE $-$ Pull an arm selected uniformly at random among the arms in $I_k$. Remove this arm from $I_k$. If $I_k$ is empty, remove $I_k$ from the set of alive intervals.
\ENDFOR
\end{algorithmic}
\end{algorithm}

In order to bound the regret of UCBF, we show that it can be decomposed into the sum of a discretization term and of the cost of learning on a finite multi-armed bandit. First, we discuss the optimal number of intervals $K$. In a second time, we present new arguments for bounding more tightly the discretization error. Then, we show that by contrast to the classical CAB, the contribution of slightly sub-optimal arms to the regret is much more limited in F-CAB problems, before obtaining a high-probability bound on the regret of our algorithm.

By dividing the continuous space of covariates into intervals, we approximate the $\mathcal{X}$-armed setting with a finite multi-armed bandit problem, which we define bellow.

\textbf{The Finite Multi-armed Bandit (F-MAB) : } 
An agent  is given a budget $T$ and a set of $K$ arms. At each step, the agent pulls an arm $k_t$ and receives a reward $y_{t}$ sampled independently with mean $m_{k_t}$. Each arm $k \in \{1,..., K\}$ can only be pulled a finite number of time, denoted $N_k$, before it dies. The agent aims at maximising the sum of her rewards.

The approximation of the $N_k$ arms in an interval $I_k$ as a single arm that can be pulled $N_k$ times is done at the price of a discretization error, as we are now forced to treat all arms in the same interval equally, regardless of possible differences of rewards within an interval. The choice of the number of intervals $K$ determines both the cost of this approximation, and the difficulty of the F-MAB problem. To analyse the dependence of those quantities on $K$, we introduce the strategy of an oracle agent facing the corresponding F-MAB problem (i.e., of an agent knowing in hindsight the expected mean rewards $m_k = \int_{I_k} m(a) da$ for pulling an arm in any interval $I_k$, and treating all arms in the same interval equally). We denote this strategy  by $\phi^d$. Assume, for the sake of simplicity, that the intervals $I_1, ..., I_K$ have been reordered by decreasing mean reward, and that there exists $f \in \{1, ..., K\}$ such that $T = N_1 + ... + N_f$. Then, $\phi^d$ pulls all arms in the intervals $I_1$ up to $I_f$. 

We can conveniently rewrite the regret $R_T$ as the sum of the regret of $\phi^d$, and of the difference between the cumulative rewards of $\phi^d$ and that of the strategy $\phi$ :

\begin{equation}\label{eq:dec_discr+learning}
     R_T = \underbrace{\underset{1 \leq t \leq T }{\sum} m\left(a_{\phi^*(t)}\right) - \underset{1 \leq t \leq T}{\sum} m\left(a_{\phi^d(t)}\right)}_{\displaystyle R_T^{(d)}} +\underbrace{ \underset{1 \leq t \leq T}{\sum} m\left(a_{\phi^d(t)}\right) - \underset{1 \leq t \leq T}{\sum} m\left(a_{\phi(t)}\right)}_{\displaystyle  R_T^{(FMAB)}}.
\end{equation}

The regret $R_T^{(d)}$ is the regret suffered by an agent with hindsight knowledge of the expected mean rewards for the different intervals. It can be viewed as the discretization error. The additional regret $R_T^{(FMAB)}$ corresponds to the cost of learning in a F-MAB setting. All arms in an interval $I_k$ have a reward close to $m_k$, so by definition of $\phi^d$ 

\begin{eqnarray}
    R_T^{(FMAB)}&\approx& \underset{k \leq f}{\sum} (N_k - n_k(T))m_k - \underset{k > f}{\sum} n_k(T) m_k.\label{eq:learning_cost}
\end{eqnarray} 

where we recall that $N_k$ denotes the number of arms belonging to interval $I_k$, and $n_k(T)$ denotes the number of arms pulled in this interval by UCBF at time $T$.
 
Choosing the number of intervals thus yields the following tradeoff : a low value of $K$ implies an easier F-MAB problem and a low value of $R_T^{(FMAB)}$, while a high value of $K$ allows for reduction of the discretization error. In finite bandits, exploration is limited : indeed, when increasing the number of intervals in a F-CAB setting, we simultaneously reduce the number of arms in each interval, and we may become unable to differentiate the mean rewards of two intervals close to the threshold $M$. Under the weak Lipschitz assumption, gaps between the rewards of two adjacent intervals are of the order $1/K$. Classical results indicate that $K^2$ pulls are needed to differentiate the mean rewards of those intervals. On the other hand, under Assumption \ref{hyp:uniforme}, the number of arms in each interval is of the order $N/K$. Thus, choosing $K$ larger than $N^{1/3}$ will only increase the difficulty of the multi-armed problem, without reducing the discretization error (since $K^2 \geq N/K$ when $K\geq N^{1/3}$).

\subsection{Bounding the discretization error}

\label{subsec:discretization}

Equation \eqref{eq:dec_discr+learning}  indicates that the regret can be decomposed as the sum of a discretization error and of the regret on the corresponding multi-armed bandit. In order to bound this discretization error, usual methods from continuum-armed bandits rely on bounding the difference between the expected reward of an arm and that of its interval by $L/K$. Thus, at each step, an algorithm knowing only the best interval may suffer a regret of the order $O(1/K)$, and the difference between the cumulative rewards of $\phi^d$ and $\phi^*$ is of the order $O(T/K)$. This argument yields sub-optimal bounds in F-CAB problems: indeed, the majority of the terms appearing in $R_T^{(d)}$ are zero, as $\phi^*$ and $\phi^{(d)}$ mostly select the same arms. 

To obtain a sharper bound on the discretization error $R_T^{(d)}$, we analyse more carefully the difference between those strategies. More precisely, we use concentrations arguments  to show that under Assumption \ref{hyp:uniforme}, $m(a_{\phi^*(T)})$ and $m_f$ are close to $M$. This result implies that under the weak Lipschitz assumption, for any pair of arms $(a_i, a_j)$ respectively selected by $\phi^*$ but not by $\phi^d$ and vice versa, $m(a_i) - m(a_j) = O(L/K)$. Finally, the margin assumption allows us to bound the number of those pairs, thus proving the following Lemma.

\begin{lem}\label{lem:bound_Rd}
Assume that $K \leq N^{2/3}$ and $K > p^{-1} \lor (1-p)^{-1}$. Under Assumptions \ref{hyp:uniforme}, \ref{hyp:lip} and \ref{hyp:beta}, there exists a constant $C_{L,Q}$ depending on $L$ and $Q$ such that with probability larger than $1 - 6e^{-2N/K^2} - 2e^{-N^{1/3}/3}$,
\[ R_T^{(d)} \leq C_{L,Q}\frac{T}{pK^2}.\]
\end{lem}

We underline that this discretization error is lower than the unavoidable error of order $O(T/K)$, encountered in classical CAB settings.

\subsection{Upper bound on the regret of UCBF}
\label{subsec:upper_bound}

Before stating our result, we bound the regret due to slightly sub-optimal arms. It is known that in the classical CAB model, slightly sub-optimal arms contribute strongly to the regret, as any agent needs at least $O(\Delta^{-2})$ pulls to detect an interval sub-optimal by a gap $\Delta$. When $\Delta$ is smaller than $\sqrt{1/T}$, the agent spends a budget proportional to $T$ to test whether this interval is optimal or not, which leads to regret of the order $O(\Delta T)$. By contrast, in a F-CAB setting, pulling arms from an interval sub-optimal by a gap $\Delta$ until it dies, contributes to the regret by a factor at most $\Delta N/K$. Under Assumptions \ref{hyp:uniforme}, \ref{hyp:lip} and \ref{hyp:beta}, the number of intervals with mean rewards sub-optimal by a gap smaller than $\Delta$ is $O(K\Delta)$. Thus, we are prevented from mistakenly selecting those slightly sub-optimal intervals too many times. This is summarised in the following remark.

\begin{remark}\label{rem:finiteness}
Under hypothesis \ref{hyp:uniforme}, \ref{hyp:lip} and \ref{hyp:beta}, intervals sub-optimal by a gap $\Delta$ contribute to the regret by a factor at most $O(\Delta^2T/p)$.
\end{remark}

Remark \ref{rem:finiteness} along with Lemma \ref{lem:bound_Rd} help us to bound  with high probability the regret of Algorithm UCBF for any mean payoff function $m$ satisfying Assumptions \ref{hyp:lip} and \ref{hyp:beta}, for the choice $K = \lfloor N^{1/3}\log(N)^{-2/3}\rfloor$ and $\delta = N^{-4/3}$. The proof of Theorem \ref{thm:bound_high_prob} is deferred to the Appendix.

\begin{thm}\label{thm:bound_high_prob}
Assume that $\lfloor N^{1/3}\log(N)^{-2/3}\rfloor> p^{-1} \lor (1-p)^{-1}$. Under Assumption \ref{hyp:uniforme}, \ref{hyp:lip} and \ref{hyp:beta}, there exists a constant $C_{L,Q}$ depending only on $L$ and $Q$ such that for the choice $K = \lfloor N^{1/3}\log(N)^{-2/3}\rfloor$ and $\delta = N^{-4/3}$,
\begin{eqnarray*}
    R_T \leq C_{L,Q} \, (T/p)^{1/3}\log(T/p)^{4/3}
\end{eqnarray*}
with probability at least $1 - 12(N^{-1} \lor e^{-N^{-1/3}/3})$.
\end{thm}

\vspace{-0.2cm}
\begin{skproof}
We use Lemma \ref{lem:bound_Rd} to bound the discretization error $R_T^{(d)}$. The decomposition in Equation \eqref{eq:dec_discr+learning} shows that it is enough to bound $R_T^{(FMAB)}$.
Recall that $\phi^d$ pulls all arms in the intervals $I_1$, $I_2$, up to $I_f$, while UCBF pulls $n_k(T)$ arms in all intervals $I_k$. Using Equation \eqref{eq:learning_cost}, we find that
\begin{eqnarray}
    R_T^{(FMAB)} &\approx& \underset{k \leq f}{\sum} (N_k - n_k(T)) (m_k-M) + \underset{k > f}{\sum} n_k(T) (M-m_k)\label{eq:fmab}\nonumber
\end{eqnarray} 
where we have used that $\sum_{k \leq f} N_k = T = \sum_{k \leq K} n_k(T)$, which in turns implies $\sum_{k \leq f} N_k - n_k(T)  = \sum_{k > f} n_k(T)$. 

On the one hand, $R_{subopt} = \sum_{k > f} n_k(T) (M-m_k)$ corresponds to the regret of pulling arms in sub-optimal intervals. We use Remark \ref{rem:finiteness} to bound the contribution of intervals sub-optimal by a gap $O(1/K)$ by a factor of the order $O(T/(pK^2))$. Classical bandit technics allow to bound the contribution of the remaining sub-optimal intervals : under Assumptions \ref{hyp:uniforme}-\ref{hyp:beta}, they contribute to the regret by a term $O(K\log(T)\log(K))$. Thus, for the choice $K = N
^{1/3}\log(N)^{-2/3}$, we can show that $R_{subopt} = O((T/p)^{1/3}\log(T/p)^{4/3})$.

On the other hand, the term $R_{opt} = \sum_{k \leq f} (N_k - n_k(T)) (m_k-M)$ is specific to finite bandit problems. The following argument shows that UCBF kills the majority of optimal intervals, and that optimal intervals $I_k$ alive at time $T$ are such that $f-k$ is bounded by a constant.

Let $I_k$ be an interval still alive at time $T$ such that $m_k > M$. Then the interval $I_k$ was alive at every round, and any interval selected by $\phi$  must have appeared as a better candidate than $I_k$. Using the definition of UCBF and Assumptions \ref{hyp:beta}, we can show that the number of arms pulled from intervals with mean reward lower than $m_k$ is bounded by a term $O(N/K + K^2\log(T))$. 

Since $T = N_1 + ... + N_f$ arms are pulled in total, the number of arms pulled from intervals with mean reward lower than $m_k$ is at least $T - (N_1 + ... + N_k) = N_{k+1} + ... + N_f \approx (f-k)N/K$. Therefore, no interval $I_k$ such that $(f-k)N/K \geq O(N/K + K^2\log(T))$ can be alive at time $T$. For the choice of $K$ described above, $(1 + K^3\log(T)/N)$ is upper bounded by a constant. Thus, there exists a constant $C>0$ such that for all $k \leq f - C$, all intervals $I_k$ have died before time $T$. We note that the number of arms in any interval is of the order $N/K$, so $R_{opt} = \sum_{f - C \leq k \leq f} (N_k - n_k(T)) (m_k-M) \leq C(m_{(f-C)}-M) N/K$. To conclude, we use Assumption \ref{hyp:lip} to show that $m_{(f-C)}-M = O(CL/K)$, and find that $R_{opt} = O(N/K^2) = O(T/(pK^2))$.
\end{skproof}

Under Assumptions similar to \ref{hyp:lip} and \ref{hyp:beta}, \cite{Auer_continuum} show that the regret of UCBC in CAB problems is $O(\sqrt{T}\log(T))$ for the optimal choice $K = \sqrt{T}/\log(T)$. By contrast, in the F-CAB problem, Theorem \ref{thm:bound_high_prob} indicates that when $p$ is a fixed constant, i.e. when the number of arms is proportional to the budget, the optimal choice for $K$ is of the order $T^{1/3}\log(T)^{-2/3}$ and the regret scales as $O(T^{1/3}\log(T)^{4/3})$. In this regime, regrets lower than that in CAB settings are thus achievable. As $N \to \infty$ and $p \to 0$, both the regret and the optimal number of intervals increase. To highlight this phenomenon, we consider regimes where $T = 0.5N^{\alpha}$ for some $\alpha \in [0,1]$ (the choice $T\leq0.5N$ reflects the fact that we are interested in settings where $T$ may be small compared to $N$, and is arbitrary). Theorem \ref{thm:bound_high_prob} directly implies the following Corollary.

\begin{cor}\label{cor:transition} Assume that $T = 0.5N^{\alpha}$ for some $\alpha \in (2/3 + \epsilon_N, 1]$, where we define $\epsilon_N = \left(\frac{2}{3}\log\log(N) + \log(2)\right) /\log(N)$. Then, for the choice $\delta = N^{-4/3}$ and  $K = \lfloor \alpha^{2/3} (2T)^{1/(3\alpha)}\log(2T)^{-2/3}\rfloor$, with probability at least $1 - 12(N^{-1} \lor e^{-N^{-1/3}/3})$ , 
$$R_T \leq C_{Q,L} T^{1/(3\alpha)}\log(T)^{4/3}$$
for some constant $C_{Q,L}$ depending on $Q$ and $L$.
\end{cor}
Corollary \ref{cor:transition} indicates that as $\alpha$ decreases, the regret increases progressively from a F-CAB regime to a CAB regime. When the budget is a fixed proportion of the number of arms, the regret scales as $O(T^{1/3}\log(T)^{4/3})$ for the optimal number of intervals $K$ of the order $T^{1/3}\log(T)^{1/2}$. As $p$ decreases and $\alpha \in (2/3 + \epsilon_N, 1]$, the regret increases as $O(T^{1/(3\alpha)}\log(T)^{4/3})$ for $K$ of the order $T^{1/(3\alpha)}\log(T)^{-2/3}$. In the limit $\alpha \to 2/3 + \epsilon_n$, the regret rate becomes $R_T = O(\sqrt{T}\log(T))$ for the optimal number of intervals of the order $\sqrt{T}/\log(T)$, which corresponds to their respective values in the CAB setting.

To understand why $\alpha = 2/3 + \epsilon_N$ corresponds to a transition from a F-CAB to a CAB setting, note that $\alpha = 2/3 + \epsilon_N$ implies $T = N/K$ : in other words, the budget becomes of the order of the number of arms per interval.  Thus, when $\alpha >  2/3 + \epsilon_N$, the oracle strategy exhausts all arms in the best interval, and it must select arms in intervals with lower mean rewards. In this regime, we see that the finiteness of the arms is indeed a constraining issue. On the contrary, if $\alpha \leq  2/3 + \epsilon_N$, no interval is ever exhausted. The oracle strategy only selects arms from the interval with highest mean reward, and our problem becomes similar to a CAB problem. Finally, we underline that when $\alpha \leq 2/3 + \epsilon_N$ the analysis becomes much simpler. Indeed, results can be directly inferred from \cite{Auer_continuum} by noticing that no interval is ever exhausted, and that Algorithm UCBF is therefore a variant of Algorithm UCBC. In this case, the optimal choice for the number of intervals remains $K = \sqrt{T}/\log(T)$, and yields a regret bound $R_T = O(\sqrt{T}\log(T))$.

\section{A lower bound}
\label{sec:lower_bound}
A careful analysis of the proof of Theorem \ref{thm:bound_high_prob} reveals that all intervals with mean reward larger than $M$ plus a gap $O(L/K)$ have died before time $T$. On the other hand, all intervals with mean rewards lower than $M$ minus a gap $O(L/K)$ have been selected but a logarithmic number of times. In other words, the algorithm UCBF is able to identify the set corresponding to the best $p$ fraction of the rewards, and it is only mistaken on a subset of measure $O(1/K)$ corresponding to arms $a_i$ such that $\vert m(a_i) - M \vert =  O(1/K)$. We use this remark to derive a lower bound on the regret of any strategy for mean payoff function $m$ in the set $\mathcal{F}_{p, L,Q}$ defined bellow.
\begin{definition}
For $p \in (0,1)$, $L>0$ and $Q>0$, we denote by $\mathcal{F}_{p, L,Q}$ the set of functions $m :[0,1] \rightarrow [0,1]$ that satisfy Equations \eqref{eq:Lipshitz} and \eqref{eq:margin}.
\end{definition}

To obtain our lower bound, we construct two functions $m_1$ and $m_2$ that are identical but on two intervals, each one of length $N^{-1/3}$. On those intervals, $m_1$ and $m_2$ are close to the threshold $M$ separating rewards of the fraction $p$ of the best arms from the rewards of the remaining arms. One of these intervals corresponds to arms with reward above $M$ under the payoff function $m_1$ : more precisely, on this interval $m_1$ increases linearly from $M$ to $M + 0.5LN^{-1/3}$, and decreases back to $M$. On this interval, $m_2$ decreases linearly from $M$ to $M - 0.5LN^{-1/3}$, and increases back to $M$. We define similarly $m_1$ and $m_2$
on the second interval by exchanging their roles, and choose the value of $m_1$ and $m_2$ outside of those intervals so as to ensure that both functions belong to the set $\mathcal{F}_{L,Q}$ for some $Q$ large enough.

Now, any reasonable strategy pulls arms in both intervals until it is able to differentiate the two mean reward functions, or equivalently until it is able to determine which interval contains optimal arms. As the average payments of those two intervals differ by $\Omega(N^{-1/3})$, this strategy must pull $\Omega(N^{2/3})$ arms in both intervals. This is possible since there are $N^{2/3}$ arms in each interval. Since arms in one of those intervals are sub-optimal by a gap of the order $N^{-1/3}$, this strategy suffers a regret $\Omega(N^{1/3})$.

In order to formalise this result, we stress the dependence of the regret on the strategy $\phi$ and the mean reward function $m$ by denoting it $R_T^{\phi}(m)$. Our results are proved for reward $y$ that are Bernoulli random variables (note that this is a special case of the F-CAB problem).

\begin{hyp}\label{hyp:Bernoulli}
For $i \in \{1, ..., N\}$,  $y_i \sim$ Bernoulli$(m(a_i))$.
\end{hyp}

In order to simplify the exposition of our results, we assume that the arms $a_i$ are deterministic.

\begin{hyp}\label{hyp:a_lb}
For $i \in \{1, ..., N\}$,  $a_i = \frac{i}{N}$.
\end{hyp}

\begin{thm}\label{thm:lower_bound} For all $p \in (0,1)$, all $L>0$, all $Q>(6/L\lor12)$, there exists a constant $C_{L}$ depending on $L$ such that under Assumptions \ref{hyp:a_lb} and \ref{hyp:Bernoulli}, for all $N \geq C_{L}(p^{-3}\lor(1-p)^{-3})$,
\begin{eqnarray*}
    \underset{\phi}{\inf} \underset{m \in \mathcal{F}_{p,Q,L}}{\sup} \mathbb{P} \left(R_T^{\phi}(m) \geq 0.01 T^{1/3}p^{-1/3}\right) \geq 0.1.
\end{eqnarray*}
\end{thm}
Theorem \ref{thm:lower_bound} shows that the bound on the regret of UCBF obtained in Theorem \ref{thm:bound_high_prob} is minimax optimal up to a polylogarithmic factor. The proof of Theorem \ref{thm:bound_high_prob} is deferred to the Appendix. Again, we stress the dependence of this regret bound on $T$ by considering regimes where $T = 0.5N^{\alpha}$. The following Corollary follows directly from Theorem \ref{thm:lower_bound}.
\begin{cor}\label{cor:lower_bound} For all $L>0$, $Q>(6/L\lor12)$, there exists a constant $C_{L}$  depending on $L$ such that such that  for all $N>\exp(3C_{L})$ and all $T$ such that $T = 0.5N^{\alpha}$ for some $\alpha \in (2/3 + C_{L}/\log(N), 1]$, under Assumptions \ref{hyp:a_lb} and \ref{hyp:Bernoulli}, 
\begin{eqnarray*}
    \underset{\phi}{\inf} \underset{m \in \mathcal{F}_{0.5N^{\alpha-1},Q,L}}{\sup} \mathbb{P} \left(R_T^{\phi}(m) \geq 0.01 T^{1/(3\alpha)}\right) \geq 0.1.
\end{eqnarray*}
\end{cor}
\section{Discussion} \label{sec:discussion}
We have introduced a new model for budget allocation with short supply when each action can be taken at most once, and side information is available on those actions. We have shown that, when covariates describing those actions are uniformly distributed in $[0,1]$, the expected reward function $m$ satisfies Assumption \ref{hyp:lip} and \ref{hyp:beta}, and the budget is proportional to the number of arms, then the optimal choice of number of intervals $K$ is of the order $T^{1/3}\log(T)^{-2/3}$, and the regret is $O(T^{1/3}\log(T)^{4/3})$. Our lower bound shows that this rate is sharp up to poly-logarithmic factors.

    Those results can readily be generalized to $d$-dimensionnal covariates. Assume that $m : [0,1]^d \rightarrow [0,1]$ is such that the weak Lipschitz assumption \ref{hyp:lip} holds for the euclidean distance, and that the margin assumption \ref{hyp:beta} is verified. Then, if $a_i \overset{i.i.d}{\sim} \mathcal{U}\left([0,1]^d\right)$, we can adapt UCBF by dividing the space $[0,1]^d$ into $K^d$ boxes of equal size. Looking more closely at our methods of proof, we note that the discretization error $R_T
^{(d)}$ remains of order $O(T/K
^2)$, while the cost of learning $R_T^{(FMAB)}$ is now bounded by $K^{d}\log(T)\log(K)$. Thus, the optimal number of intervals $K$ is of the order $T^{1/(d+2)}\log(T)^{-2/(d+2)}$, and the regret is of the order $O(T^{d/(d+2)}\log(T)^{4/(d+2)})$. We refer the interested reader to the Appendix, where precise statements of our hypotheses and results are provided, along with a description of the extension of Algorithm UCBF to multi-dimensional covariates.

\section*{Broader impact}
We present an algorithm for the problem of allocating a limited budget among competing candidates. This algorithm is easy to implement, and enjoys strong theoretical guarantees on its performance making it attractive and reliable in relevant applications. Nevertheless, we emphasise that the considered framework is based on the premise that the decision-maker is purely utility-driven. We leave it to the decision-maker to take additional domain-specific considerations into account.

\section*{Acknowledgement}
We would like to thank Christophe Giraud, Vianney Perchet and Gilles Stoltz for their stimulating suggestions and discussions.

\bibliographystyle{apalike}
\bibliography{refFinite}


\newpage
\appendix
\section*{Appendix}

Theorem \ref{thm:bound_high_prob} is proved Section \ref{subsec:TH1}, and Theorem \ref{thm:lower_bound} is proved Section \ref{subsect:TH2}. Section \ref{sec:highdim} is dedicated to stating and proving an upper bound on the regret of UCBF in higher dimension. Lemmas used in those Sections are proved in Section \ref{subsec:auxi}. First, let us state the following Lemma, which controls the fluctuations of $m$ within an interval.

\begin{lem}\label{lem:lipschitz}
Let $a \in [0,1]$ be such that $m(a) = M + \alpha L/K$  for some $\alpha >0$. Moreover, let $k$ be such that $a \in I_k$. Then $$\max_{a' \in I_k} m(a') \leq M + \left(\alpha + (\alpha \lor 1)\right)\frac{L}{K},$$ and $$\min_{a' \in I_k} m(a') \geq M + \left(\alpha - \frac{(\alpha \lor 2)}{2}\right)\frac{L}{K}.$$
Similarly, let $a \in [0,1]$ be such that $m(a) = M - \alpha \frac{L}{K}$, where $\alpha >0$. Moreover, let $k$ be such that $a \in I_k$. Then $$\min_{a' \in I_k} m(a') \geq M - \left(\alpha + (\alpha \lor 1)\right)\frac{L}{K},$$ and $$\max_{a' \in I_k} m(a') \leq M - \left(\alpha - \frac{(\alpha \lor 2)}{2}\right)\frac{L}{K}.$$
\end{lem}

\section{Proof of Theorem \ref{thm:bound_high_prob}}
\label{subsec:TH1}

To prove Theorem \ref{thm:bound_high_prob}, we show that the regret $R_T$ can be decomposed as the sum of a discretization error term and of a term corresponding to the regret of pulling a game of finite bandit with $K$ arms. To do so, we introduce further notations.

Recall that for $k = 1, ..., K$, $m_k = K\int_{a \in I_k}m(a)da$ is the mean payment for pulling an arm uniformly in interval $I_k$. In order to avoid cumbersome notations for reordering the intervals, we assume henceforth (without loss of generality) that $\left\{m_k\right\}_{1\leq k \leq K}$ is a decreasing sequence. 

If we knew the sequence $\left\{m_k\right\}_{1\leq k \leq K}$ but not the reward of the arms $m(a_i)$, a reasonable strategy would be to pull all arms in the fraction $p$ of the best intervals, and no arm in the remaining intervals.  If all intervals contained the same number of arms $N/K$, we would pull all arms in the interval $I_1$, $I_2$, up to $I_f$, where $f = \lfloor pK\rfloor$, and we would pull the remaining arms randomly in $I_{f+1}$. Note however that since the arms are randomly distributed, the number of arms in each interval varies. Thus, a good strategy if we knew the sequence $\left\{m_k\right\}_{1\leq k \leq K}$ would consist in pulling all arms in the intervals $I_1$, $I_2$, up to $I_{\widehat{f}}$, where $\widehat{f}$ is such that $ N_{1} + .. +  N_{\widehat{f}}< T \leq N_{1} + .. +  N_{\widehat{f}+1}$, and pull the remaining arms in $I_{\hat{f}+1}$. We call this strategy "oracle strategy for the discrete problem", and we denote it $\phi^d$. Recall that we denote by $\phi^*(t)$ the arm pulled at time $t$ by the oracle strategy, and by $\phi(t)$ the arm pulled at time $t$ by UCBF.

We decompose $R_T$ as follows :
\begin{eqnarray*}
    R_T &=& \underset{t = 1.. T }{\sum} m(a_{\phi^*(t)}) - \underset{t = 1..T}{\sum} m(a_{\phi(t)})\\
     &=& \underset{t = 1.. T }{\sum} m(a_{\phi^*(t)}) - \underset{t = 1..T}{\sum} m(a_{\phi^d(t)}) + \underset{t = 1.. T }{\sum} m(a_{\phi^d(t)}) - \underset{t = 1..T}{\sum} m(a_{\phi(t)}).
\end{eqnarray*}

Let $R_T^{(d)} = \underset{t = 1.. T }{\sum} m(a_{\phi^*(t)}) - \underset{t = 1..T}{\sum} m(a_{\phi^d(t)})$. By definition, $R_T^{(d)}$ is the regret of the oracle stratgey for the discrete problem, and corresponds to a discretization error. We bound this term in Section \ref{subsubsec:discretise}. 

Let $R_T^{(FMAB)} = \underset{t = 1.. T }{\sum} m(a_{\phi^d(t)}) - \underset{t = 1..T}{\sum} m(a_{\phi(t)})$ be the regret of our strategy against the oracle strategy for the discrete problem. $R_T^{(FMAB)}$ corresponds to the regret of the corresponding finite $K$-armed bandit problem. A bound on this term is obtained in Section \ref{subsubsec:k_arm}.
\subsection{Bound on the discretization error $R_T^{(d)}$ and proof of Lemma \ref{lem:bound_Rd}}
\label{subsubsec:discretise}

To bound the discretization error $R_T^{(d)}$, we begin by controlling the deviation of $\widehat{f}$ and $m_{\hat{f}}$ from their theoretical counterparts $f$ and $M$.

\begin{lem}\label{lem:control_hat_f}
With probability at least $1 - 4e^{-\frac{2N}{K^2}}$, we have $\vert\widehat{f} - f \vert \leq 1$. On this event, $\left \vert m_{\hat{f}} - M \right \vert \leq 4L/K$ and $m_{\widehat{f}+1} \in [M - 8L/K, M + L/K]$.
\end{lem}

Then, we define $\widehat{M} = m(a_{\phi*(T)})$ and control its deviation from $M$.

\begin{lem}\label{lem:control_hat_M}
Assume that $p \in (1/K, 1-1/K)$. Then,
with probability at least $1 - 2e^{-\frac{2N}{K^2}}$, we have $\vert \widehat{M} - M \vert \leq L/K$.
\end{lem}

We show later that with high probability, $\phi^*$ and $\phi^d$ may only differ on arms $i$ such that $m(a_i) \in [M- 16L/K, M+L/K]$. The following lemma controls the number of those arms.

\begin{lem}\label{lem:ech_arms} Assume that $K\leq N^{2/3}$. Then,
with probability at least $1 - 2e^{-\frac{N^{1/3}}{3}}$, 
$$\left \vert \left\{i : m(a_i) \in \left[M - \frac{16L}{K}, M+\frac{L}{K}\right]\right\}\right\vert \leq \frac{32LQ N}{K}.$$
\end{lem}

Using Lemmas \ref{lem:control_hat_f}-\ref{lem:ech_arms}, we control the discretization cost $R_T^{(d)}$ on the following event. Let
\begin{eqnarray*}\label{lem:controle_R1}
    \mathcal{E}_{a} & =& \left\{ \vert\widehat{f} - f \vert \leq 1\right\}\cap \left\{\vert \widehat{M} - M \vert \leq L/K \right\} \\
    &\cap& \left\{ \left \vert \left\{i : m(a_i) \in \left[M - \frac{16L}{K}, M+\frac{L}{K}\right]\right\}\right\vert \leq \frac{32LQ N}{K}\right\}.
\end{eqnarray*}

Note that under the assumptions of Lemmas \ref{lem:control_hat_M}-\ref{lem:ech_arms},  $\mathbb{P}\left(\mathcal{E}_a\right) \geq 1 - 6e^{-\frac{2N}{K^2}} - 2e^{-\frac{N^{1/3}}{3}}$ by Lemma \ref{lem:control_hat_f}-\ref{lem:ech_arms}. Moreover on $\mathcal{E}_a$, $\vert m_{\hat{f}}-M\vert \leq 4L/K$ and $m_{\widehat{f}+1} \in [M - 8L/K, M + L/K]$.

\begin{lem}\label{lem:E_a}
On the event $\mathcal{E}_{a}$, $R_T^{(d)} \leq \frac{384QL^2N}{K^2}$.
\end{lem}

Lemma \ref{lem:bound_Rd} follows from Lemma \ref{lem:control_hat_f}, Lemma \ref{lem:control_hat_M}, Lemma \ref{lem:ech_arms}  and Lemma \ref{lem:E_a}.

\subsection{Bound on the regret of the discrete problem $R_T^{(FMAB)}$}
\label{subsubsec:k_arm}

We bound $R_T^{(FMAB)}$ on a favourable event, on which both the number of arms in each interval and the payment obtained by pulling those arms do not deviate too much from their expected value. Under Assumption \ref{hyp:uniforme}, $\mathbb{E}\left[N_k\right] = N/K$ for all $k  = 1, ..., K$. The following Lemmas provides a high probability bound on $\underset{k = 1, ..., K}{\max} \vert N_k - N/K \vert$.

\begin{lem}\label{lem:equi_nk}
Assume that $K \leq N^{2/3}/4$. Then, 
\begin{eqnarray*}
\mathbb{P}\left(\underset{k \in \{1,..,K\}}{\max} \left \vert N_k - \frac{N}{K} \right \vert \geq \frac{N}{2K}\right) \leq   2Ke^{-\frac{N^{1/3}}{3}}.
\end{eqnarray*}
\end{lem}

Now, we show that on an event of large probability, for $k = 1, ..., K$ and $s \leq (N_k \land T)$, $\widehat{m}_k(s)$ does not deviate of $m_k$ by more that $\sqrt{\log(T/\delta)/2s}$. 

Let $k \in \{1,...,K\}$ be such that $N_k > 0$. For $s \leq n_k(T)$, we denote by $\pi_k(s)$ the $s$-th armed pulled in interval $I_k$ by UCBF. With these notations, for all $s = 1, ..., n_k(T)$, $\widehat{m}_k(s)$ is defined by UCBF as $\widehat{m}_k(s) = \frac{1}{s}\underset{t = i = 1, ..., s}{\sum} y_{\pi_k(i)}$. We define similarly $\pi_k(s)$ for $s \in [n_k(T)+1,N_k]$ by selecting uniformly at random without replacement the remaining arms in $I_k$, and let $\widehat{m}_k(s) = \frac{1}{s}\underset{t = i = 1, ..., s}{\sum} y_{\pi_k(i)}$ for $s = n_k(T) + 1, ..., N_k$, and $\widehat{m}_k(0) = 0$.

\begin{lem}\label{lem:borne_m} 
\begin{eqnarray*}
\mathbb{P}\left( \exists k \in \{1,...,K\}, s \leq (N_k\land T) : \left\vert \widehat{m}_k(s) - m_k \right\vert \geq \sqrt{\frac{\log(T/\delta)}{2s}}\right) &\leq& 2K\delta.
\end{eqnarray*}
\end{lem}

Then, we define
\begin{eqnarray*}
\mathcal{E}_b = && \left\{\underset{k = 1..K}{\cap} \left\{N_k \in \left[\frac{N}{2K}, \frac{3N}{2K}\right]\right\}\right\}\\
&& \cap \left\{\underset{k = 1 .. K}{\cap} \underset{s = 1..(N_k\land T)}{\cap} \left\{ \vert m_k - \widehat{m}_k(s)\vert \leq \sqrt{\frac{\log(T/\delta)}{2s}} \right\}\right\}.
\end{eqnarray*}
Combining Lemma \ref{lem:equi_nk} and Lemma \ref{lem:borne_m}, we find that when $K \leq N^{2/3}/4$, $$\mathbb{P}\left( \mathcal{E}_b\right) \geq 1 - 2Ke^{-\frac{N^{1/3}}{2}} - 2K\delta.$$

Now, we decompose $R_T^{(FMAB)}$ in the following way. Recall that
\begin{align}
R_T^{(FMAB)} =& \underset{t = 1...T}{\sum} m(a_{\phi^d(t)}) - \underset{t = 1...T}{\sum} m(a_{\phi(t)}) \nonumber .
\end{align}
Recall that $\phi^d$ pulls all arms in the interval $I_1$, ..., $I_{f}$, and pull the remaining arms in the interval $I_{f+1}$. In the following, we denote $\Phi^{d}(T)$ the set of arm pulled by $\phi^d$ at time $T$. Thus, 
\begin{align}
\underset{t = 1...T}{\sum} m(a_{\phi^d(t)}) =& \underset{k = 1..\widehat{f}\ }{\sum} \underset{\ a_i \in I_k}{\sum} m(a_i) + \underset{a_i \in \Phi^d(T)\cap I_{\hat{f}+1}}{\sum} m(a_{i}) \nonumber.
\end{align}
The number of arms pulled by $\phi^d$ is equal to $T$, and we can write
\begin{align}
\underset{t = 1...T}{\sum} m(a_{\phi^d(t)}) =& \underset{k = 1..\widehat{f}\ }{\sum} \underset{\ a_i \in I_k}{\sum} \left( m(a_i) - M\right) + \underset{a_i \in \Phi^d(T)\cap I_{\hat{f}+1}}{\sum} \left( m(a_{i}) - M\right) + TM \label{eq:decPhi}.
\end{align}
On the other hand, we decompose the total payment obtained by $\phi$ as the sum of the payment obtained by pulling arms also selected by $\phi^d$ (i.e. arms in $I_1$, ..., $I_f$ and $I_{f+1} \cap \Phi^d(T)$), and the sum of payment for pulling arms that were not selected by $\phi^d$ (i.e. arms in $I_{f+1} \cap \overline{\Phi^d(T)}$ and in $I_{f+2}$, ..., $I_k$). Recall that $\Phi(T)$ is the set of arms pulled by UCBF at time $T$.
\begin{align}
\underset{t = 1...T}{\sum} m(a_{\phi(t)}) =&\underset{k = 1.. \widehat{f}\ }{\sum}\ \underset{a_i \in I_k \cap \Phi(T)}{\sum} m(a_i) + \underset{a_{i} \in I_{\hat{f}+1} \cap \Phi(T) \cap  \Phi^d(T)}{\sum} m(a_i) \nonumber\\
 & - \left( \underset{a_i \in I_{\hat{f}+1}\cap \Phi(T) \cap \overline{\Phi^d(T)}}{\sum} - m(a_{\phi(t)})\right) -  \left(\underset{k = \widehat{f}+2 .. K\ }{\sum}\underset{a_i \in I_k \cap \Phi(T)}{\sum} - m(a_{\phi^d(t)})\right)\nonumber.
\end{align}

Again, $T$ arms are pulled by $\phi$, and we can write
\begin{align}
\underset{t = 1...T}{\sum} m(a_{\phi(t)}) =&\underset{k = 1.. \widehat{f}\ }{\sum}\ \underset{a_i \in I_k \cap \Phi(T)}{\sum} (m(a_i)-M) + \underset{a_{i} \in I_{\hat{f}+1} \cap \Phi(T) \cap  \Phi^d(T)}{\sum} (m(a_i)-M) \nonumber\\
 &- \underset{a_i \in I_{\hat{f}+1}\cap \Phi(T) \cap \overline{\Phi^d(T)}}{\sum} (M- m(a_i)) -  \underset{k = \widehat{f}+2 .. K\ }{\sum}\underset{a_i \in I_k \cap \Phi(T)}{\sum} (M- m(a_i)) + TM\label{eq:decPhid}.
\end{align}

Subtracting equation \eqref{eq:decPhid} from equation  \eqref{eq:decPhi}, we find that

\begin{align}
R_T^{(FMAB)} =& \underset{k = 1..\widehat{f}\ }{\sum} \underset{\ a_i \in I_k \cap \overline{\Phi(T)}}{\sum} \left(m(a_{i}) - M\right)  + \underset{\ a_i \in I_{\hat{f}+1} \cap \Phi^d(T) \cap \overline{\Phi(T)}}{\sum} \left(m(a_{i}) - M\right) \nonumber \\
&+ \underset{\ a_i \in I_{\hat{f}+1} \cap \Phi(T) \cap \overline{\Phi^d(T)}}{\sum} \left(M-m(a_{i}) \right)+\underset{k = \widehat{f}+2 .. K\ }{\sum}\underset{a_i \in I_k\cap \Phi(T)}{\sum} \left( M - m(a_i)\right).\nonumber
\end{align}

We write $$R_{\hat{f}+1} = \underset{\ a_i \in I_{\hat{f}+1} \cap \Phi^d(T) \cap \overline{\Phi(T)}}{\sum} \left(m(a_{i}) - M\right) \nonumber  +\underset{\ a_i \in I_{\hat{f}+1} \cap \Phi(T) \cap \overline{\Phi^d(T)}}{\sum} \left(M-m(a_{i}) \right),$$
$$R_{opt} = \underset{k = 1..\widehat{f}}{\sum}\  \underset{a_i \in I_k \cap \overline{\Phi(T)}}{\sum} \left(m(a_{i}) - M\right),$$
and $$R_{subopt} = \underset{k = \widehat{f}+2 .. K}{\sum}\ \underset{a_i \in I_k \cap \Phi(T)}{\sum} \left(M  - m(a_i)\right).$$

The decomposition $R_T^{(FMAB)} = R_{opt} + R_{\hat{f}+1} + R_{subopt}$ show that three phenomenons contribute to the regret of $\phi$ on the discrete problem. The side effect term $R_{\hat{f}+1}$ can easily be bounded : there are most $1.5N/K$ arms in $I_{\hat{f}+1}$, and so there are at most $1.5N/K$ terms in $R_{\hat{f}+1}$. On the event $\mathcal{E}_a$, $m_{\hat{f}+1} \in [M-8L/K, M+L/K]$. Using Lemma \ref{lem:lipschitz}, we see that for each arm $a_i \in I_{\hat{f}+1}$, $\vert m(a_i) - M\vert\leq 16L/K$. Thus, on $\mathcal{E}_a \cap \mathcal{E}_b$, $R_{\hat{f}+1} \leq 24N/K^2$.

Now, we say that an interval $I_k$ is sub-optimal if $m_k < m_{\hat{f}+1}$ and is optimal if $m_k \geq m_{\hat{f}}$. $R_T^{(FMAB)} - R_{\hat{f}+1}$ is the sum of a term $R_{opt}$, induced by the remaining arms in the optimal intervals, and a term $R_{subopt}$, induced by pulls of arms in sub-optimal intervals. The following Lemma will be used to control those terms.

For two intervals $I_k$, $I_l$ such that $m_k > m_l$, we provide a bound on the number of arms drawn in $I_l$ given that there are still arms available in the better interval $I_k$. For two intervals $k,l \in \{1,...,K\}^2$, we  denote henceforth $\Delta_{k,l} = m_k - m_l$.

\begin{lem}\label{lem:bound_nl}
Let $k\in \{1,...,K\}$. On the event $\mathcal{E}_b \cap \{n_k(T) < N_k \}$, a.s. for all intervals $I_l$ such that $\Delta_{k,l} > 0$, $n_l(T) \leq \frac{3\log(T/\delta)}{\Delta_{k,l}^2}$.
\end{lem}

To bound the regret $R_{subopt}$, we take advantage of the fact that every slightly sub-optimal interval $k$ cannot be selected more than $N_k$ times. This is done in the following lemma.
\begin{lem}\label{lem:R_subopt}
On the event $\mathcal{E}_a \cap \mathcal{E}_b$, 
\begin{equation*}
    R_{subopt} \leq  \frac{600L^2QN}{K^2} + 384\log(T/\delta)KQ\left(\log_2(K/L) \lor 1\right).
\end{equation*}
\end{lem}

While $R_{subopt}$ corresponds to the regret of pulling sub-optimal arms, and is bounded using classical bandit arguments, $R_{opt}$ corresponds to the regret of not having pulled optimal arms. We first control the number of optimal arms that have not been pulled. The arguments used to prove Lemma \ref{lem:R_subopt} can be used to control the number of arms pulled in sub-optimal intervals, which is equal to the number of non-zero terms in $R_{opt}$.

\begin{lem}\label{lem:n_subopt}
On $\mathcal{E}_a \cap \mathcal{E}_b$, the number of arms pulled in sub-optimal intervals by UCBF is bounded by $30Q(LN/K + \log(T/\delta)K^2/L)$.
\end{lem}

This number is equal to the number of optimal arms that have not been pulled, and thus to the number of non-zero terms in $R_{opt}$. Note that this number is at least of order $N/K \lor K^2$, while $R_T^{(d)} + R_{subopt}$ is of the order $N/K^2 \lor K$ . Thus, bounding each term in $R_{opt}$ by $1$ will likely lead to sub-optimal bounds on the regret $R_T$. In the next Lemma, we characterise intervals whose arms have all been pulled by UCBF. Note that those intervals do not contributes to $R_{opt}.$

\begin{lem}\label{lem:used_up}
Let $A = 35\sqrt{\frac{K^3Q\log(T/\delta)}{NL}\lor 1}$. At time $T$, on the event $\mathcal{E}_{a}\cap\mathcal{E}_{b}$, all arms in intervals $I_k$ such that $m_k \geq M + AL/K$ have been pulled.
\end{lem}

Using Lemmas \ref{lem:n_subopt} and \ref{lem:used_up}, we can finally control $R_{opt}$.

\begin{lem}\label{lem:bound_Ropt}
On event $\mathcal{E}_a \cap \mathcal{E}_b$, 
\begin{equation*}
R_{opt}\leq 60AQ\left(\frac{L^2N}{K^2}+ \log(T/\delta)K\right)
\end{equation*}
\end{lem}

To conclude, note that for the choice $\delta = N^{-4/3}$ and $K = \lfloor N^{1/3}\log(N)^{-2/3}\rfloor  \geq (1-p\land p)^{-1}$,
\begin{equation*}
    A \leq 35\sqrt{\frac{N\log(N)^{-2}\log(pN^{7/3})}{NL}\lor 1}\leq 35\sqrt{\frac{7\log(N)^{-1}}{3L}\lor 1}.
\end{equation*}
Moreover $K^{-2} = \lfloor N^{1/3}\log(N)^{-2/3}\rfloor^{-2}\leq 4\left( N^{1/3}\log(N)^{-2/3}\right)^{-2} $ since $N^{1/3}\log(N)^{-2/3} \geq 2$. Thus, on the event $\mathcal{E}_a\cap \mathcal{E}_b$, 
\begin{eqnarray*}
R_{opt}&\leq& 2100Q\left(4L^2N^{1/3}\log(N)^{4/3}+ 7/3\log(N)^{1/3}N^{1/3}\right)\sqrt{\frac{7\log(N)^{-1}}{3L}\lor 1}\\
R_{subopt}&\leq& 2400L^2QN^{1/3}\log(N)^{4/3} + 896QN^{1/3}\log(N)^{1/3}\left(\log_2(N/L) \lor 1\right)\\
R_{\hat{f}+1}&\leq&96N^{1/3}\log(N)^{4/3}\\
R_T^{(d)}&\leq& 1536QLN^{1/3}\log(N)^{4/3}.
\end{eqnarray*}

Thus, on $\mathcal{E}_a\cap \mathcal{E}_b$, we find that
$$R_T \leq CN^{1/3}\log(N)^{4/3},$$
or equivalently that 
$$R_T \leq C(T/p)^{1/3}\log(T/p)^{4/3}$$
for some constant $C$ depending only on $L$ and $Q$. Note that $K \leq N^{2/3}/4$ as soon as $N \geq 30$. Using the Lemmas \ref{lem:control_hat_f}, \ref{lem:control_hat_M}, \ref{lem:ech_arms}, \ref{lem:equi_nk} and \ref{lem:borne_m}, we find that  the event $\mathcal{E}_a\cap \mathcal{E}_b$ occurs with probability at least $1 - 6^{-2\lfloor N^{1/3}\log(N)^{4/3}\rfloor} - 2e^{-N^{1/3}/3} - 2e^{-N^{1/3}/3}N^{1/3}\log(N)^{-2/3} - 2N^{-1} \geq 1 - 12(N^{-1}\lor e^{-N^{1/3}/3})$.

\section{Proof of Theorem \ref{thm:lower_bound}}
\label{subsect:TH2}
Before proving Theorem \ref{thm:lower_bound}, we recall that under Assumption \ref{hyp:a_lb}, the set of covariates $(a_1, ..., a_N) = (1/N, ..., 1)$ is deterministic. We prove Theorem \ref{thm:lower_bound} by studying reward that are independent Bernoulli variables : under Assumption \ref{hyp:Bernoulli}, $y_i \sim$ Bernoulli$(m(a_i))$ for $i = 1,.., N$. At each time $t$, a strategy $\phi$ selects which arm $\phi(t)$ to pull based  on the past observations $(\phi(1), y_{\phi(1)}, ...,\phi(t-1), y_{\phi(t-1)})$. For $t = 1, ..., T$, let $\mathcal{H}_t = (a_1, y_1, ..., a_t, y_t)$.

Let  $m_0$ and $m_1$ be two payoff functions. We denote by $\mathbb{P}_0$ the distribution of $\mathcal{H}_T$ when the payoff function is $m_0$, and $\mathbb{P}_1$ the distribution of $\mathcal{H}_T$ when the payoff function is $m_1$. Moreover, let $\mathcal{Z}$ be any event $\sigma(\mathcal{H}_T)$-measurable. According to Bretagnolle-Huber inequality (see, e.g., Theorem 14.2 in \cite{lattimore2018bandit})

\begin{equation*}
\mathbb{P}_0(\mathcal{Z}) + \mathbb{P}_1(\bar{\mathcal{Z}}) \geq \frac{1}{2}\exp\left(-KL(\mathbb{P}_0, \mathbb{P}_1) \right).
\end{equation*}

Let us sketch the proof of Theorem \ref{thm:lower_bound}. In a first time, we design two payoff functions $m_0$ and $m_1$ that satisfy Assumptions \ref{hyp:lip} and \ref{hyp:beta} and differ on a small number of arms. Then, we bound their Kullblack-Leibler divergence. Finally, we define an event $\mathcal{Z}$ which is favorable for $m_1$ and unfavorable for $m_0$, and we provide lower bounds for $R_T$ on $\mathcal{Z}$ under $\mathbb{P}_0$ and on $\overline{\mathcal{Z}}$ under $\mathbb{P}_1$. 

We will henceforth assume that $$N \geq \frac{1}{(p\land1-p)^3(L \land 0.5)^{2}} \lor 811.$$
In order to define $m_0$ and $m_1$, we introduce the following notations. Let $\alpha \in (20 N^{-2/3},0.5]$ to be defined later, and let $\tilde{L} = L \land 0.5$ and $\delta = \alpha (N\tilde{L}^2)^{-1/3}$. Now, define $x_0 = 1-p - 2\delta$ and $x_1 = 1-p + 2\delta$. The inequality $2\delta < p\land (1-p)$ ensures that $0<x_0<1-p<x_1<1$. Moreover, $\tilde{L}(x_0\lor 1-x_1) \leq 1/2$ and $\tilde{L}\delta < 1/4$. We define $m_0$ and $m_1$ as follows.

\[ m_0(x) = \left\{
  \begin{array}{ll}
    1/2 - \tilde{L}(x_0-x) &\text{ if } x \in [0, x_0)\\
    1/2 - \tilde{L}(x-x_0) &\text{ if } x \in [x_0, x_0 + \delta)\\
    1/2 - \tilde{L}(1-p - x) &\text{ if } x \in [x_0 + \delta, 1-p)\\
    1/2 + \tilde{L}(x-(1-p)) &\text{ if } x \in [1-p, 1-p + \delta)\\
    1/2 + \tilde{L}(x_1-x) &\text{ if } x \in [1-p + \delta, x_1)\\
    1/2 + \tilde{L}(x-x_1) &\text{ if } x \in [x_1, 1]
  \end{array}
\right.\]
Define similarly
\[ m_1(x) = \left\{
  \begin{array}{ll}
    1/2 - \tilde{L}(x_0-x) &\text{ if } x \in [0, x_0)\\
    1/2 + \tilde{L}(x-x_0) &\text{ if } x \in [x_0, x_0 + \delta)\\
    1/2 + \tilde{L}(1-p - x) &\text{ if } x \in [x_0 + \delta, 1-p)\\
    1/2 - \tilde{L}(x-(1-p)) &\text{ if } x \in [1-p, 1-p + \delta)\\
    1/2 - \tilde{L}(x_1-x) &\text{ if } x \in [1-p + \delta, x_1)\\
    1/2 + \tilde{L}(x-x_1) &\text{ if } x \in [x_1, 1]
  \end{array}
\right.\]
The functions $m_0$ and $m_1$ are bounded in $[0,1]$, piecewise linear. They differ only on $[x_0, x_1]$, and are such that $$\min\left\{A : \lambda\left(\{x : m_0(x)\geq A\}\right) < p \right\} = \min\left\{A : \lambda\left(\{x : m_1(x)\geq A\}\right) < p \right\} = 1/2.$$ 
Under hypotesis \ref{hyp:a_lb}, the $T = pN$ best arms for the payoff function $m_0$ are in $[1-p, 1] \cap \{x_0\}$, while the $T = pN$ best arms for the payoff function $m_1$ are in $[x_1, 1] \cap [x_0, 1-p]$.

\begin{lem}\label{lem:h2h3}
The payoff functions $m_0$ and $m_1$ satisfy Assumptions \ref{hyp:lip} and \ref{hyp:beta}.
\end{lem}

Next, we bound the Kullback-Leibler divergence between $\mathbb{P}_0$ and $\mathbb{P}_1$.

\begin{lem}\label{lem:KL_01}
For the functions $m_0$ and $m_1$ defined above, $$KL(\mathbb{P}_0, \mathbb{P}_1) \leq  70.4\alpha^3.$$
\end{lem}

We define $\mathcal{Z}$  as the following event :

$$\mathcal{Z} = \left\{ \underset{a_i \in [ x_0, 1-p]}{\sum} \mathds{1}_{\{i \in \Phi(T)\}} \geq N\delta -2\right\}.$$

Because of Assumption \ref{hyp:a_lb}, there are between $\lfloor 2N\delta \rfloor$ and $\lceil 2N\delta \rceil$ arms in $(x_0, 1-p)$. Under $\mathbb{P}_0$, the arms in $(x_0, 1-p)$ are sub-optimal, so $\mathcal{Z}$ is disadvantageous. On the contrary, under $\mathbb{P}_1$ all arms in $(x_0, 1-p)$ are optimal under $m_1$, and so $\overline{\mathcal{Z}}$ is disadvantageous. We provide a more detailed statement in the following lemma.

\begin{lem}\label{lem:R(m)}
Under $\mathbb{P}_0$, on $\mathcal{Z}$, $R_T \geq 0.22\alpha^2N^{1/3}$. Under $\mathbb{P}_1$, on $\overline{\mathcal{Z}}$, $R_T \geq 0.22\alpha^2N^{1/3}$.
\end{lem}

Since $N \geq 811$, we can choose  for example $\alpha = 0.23$. Using $(a\lor b) \geq (a+b)/2$, we see that 
$$\max\left\{\mathbb{P}_0\left(R_T \geq 0.01N^{1/3} \right), \mathbb{P}_1\left(R_T \geq 0.01N^{1/3}\right)\right\} \geq 0.1.$$

\section{Upper bound on the regret in multi-dimensional settings}\label{sec:highdim}

In this section, we provide an upper bound on the regret of a natural extension of Algorithm UCBF to $d$-dimensional covariates. More precisely, we assume that the arms are described by covariates in the set $\mathcal{X} = [0,1]^d$ for some $d \in \mathbb{N}^*$. Similarly to the one-dimensional case, we assume that the covariates are uniformly distributed in $\mathcal{X}$:

\begin{hyp}\label{hyp:d-uniforme}
For $i = 1, ..., N$, $a_i \overset{i.i.d.}{\sim} \mathcal{U}([0,1]^d)$.
\end{hyp}

\noindent As in the one dimensional setting, we assume that the mean payoff function is weakly $L$-Lipschitz  with regard to the Euclidean distance:

\begin{hyp}\label{hyp:d-lip}
For all $(x, y) \in [0,1]^d \times [0,1]^d$, $$\left \vert m(x) - m(y) \right \vert \leq  \max\left\{ \vert M-m(x)\vert, L\left \Vert x - y \right \Vert_2\right\}.$$
\end{hyp}

Moreover we assume that the mean reward function $m : [0,1]^d \rightarrow [0,1]$ verifies Assumption \ref{hyp:beta} (here, $\lambda$ denotes the Lebesgue measure on $[0,1]^d$). Then, the UCBF Algorithm can readily be generalized to this $d$-dimensional setting, as described in Algorithm \ref{alg:d-dim}. The following Theorem bounds the regret of Algorithm $d$-UCBF.

\begin{algorithm}
\caption{$d$-dimensional Upper Confidence Bound for Finite continuum-armed bandits ($d$-UCBF)\label{alg:d-dim}}
\begin{algorithmic}
\STATE \textbf{Parameters:} $K, \delta$
\STATE \textbf{Initialisation:} Divide $[0,1]^d$ into $K^d$ bins $B_k$ such that for $k \in \{0, ..., K^d-1\}$, $B_k = [\frac{k_1}{K}, \frac{k_1+1}{K}) \times ... \times [\frac{k_d}{K}, \frac{k_d+1}{K})$, where $(k_1, ..., k_d)$ denotes the $d$-ary representation of $k$. Let $N_k = \sum_{1 \leq i \leq N} \mathds{1}\{a_i \in B_k\}$ be the number of arms in the bin $B_k$. Define the set of bins alive as the set of bins $B_k$ such that $N_k \geq 2$. Pull an arm uniformly at random in each bin alive.
\FOR{$t = K^d+1, ..., T$} 
    \STATE $-$ Select an bin $B_k$ that maximizes $\widehat{m}_{k}(n_k(t-1)) + \sqrt{\frac{\log(T/\delta)}{2n_k(t-1)}}$ among the set of alive bins, where $n_k(t-1)$ is the number of arms pulled from $B_k$ by the algorithm before time $t$, and $\widehat{m}_{k}(n_k(t-1))$ is the average reward obtained from those $n_k(t-1)$ samples.
    \STATE $-$ Pull an arm selected uniformly at random among the arms in $B_k$. Remove this arm from $B_k$. If $B_k$ is empty, remove $B_k$ from the set of alive bins.
\ENDFOR
\end{algorithmic}
\end{algorithm}

\begin{thm}\label{thm:d-bound_high_prob}
Under Assumption \ref{hyp:d-uniforme}, \ref{hyp:d-lip} and \ref{hyp:beta}, there exists a constant $C_{L,Q,p, d}$ depending only on $L$, $Q$ $p$ and $d$ such that for the choice $K = \lceil N^{\frac{1}{d+2}}\log(N)^{-\frac{2}{d+2}}\rceil$ and $\delta = N^{-\frac{2d+2}{d+2}}$,
\begin{eqnarray*}
    R_T \leq C_{L,Q,p,d} \, T^{\frac{d}{d+2}}\log(T)^{\frac{4}{d+2}}
\end{eqnarray*}
with probability $1 - O(N^{-1})$.
\end{thm}

The rest of this Section is devoted to proving Theorem \ref{thm:d-bound_high_prob}. To do so, we  follow the main lines of the proof of Theorem \ref{thm:bound_high_prob}. Some Lemmas follow readily from results developed in Section \ref{subsec:TH1}, and their proofs are therefore omitted. The remaining Lemmas are proved in Section \ref{subsec:auxi}.

Let us now prove Theorem \ref{thm:d-bound_high_prob}. As for Theorem \ref{thm:bound_high_prob}, we begin by controlling the fluctuations of the mean payoff function $m$ within a bin.


\begin{lem}\label{lem:d-lipschitz}
Let $a \in [0,1]^d$ be such that $m(a) = M + \alpha L/K$  for some $\alpha >0$. Moreover, let $k$ be such that $a \in B_k$. Then $$\max_{a' \in B_k} m(a') \leq M + \left(\alpha + (\alpha \lor \sqrt{d})\right)\frac{L}{K},$$ and $$\min_{a' \in B_k} m(a') \geq M + \left(\alpha - \frac{(\alpha \lor 2\sqrt{d})}{2}\right)\frac{L}{K}.$$
Similarly, let $a \in [0,1]^d$ be such that $m(a) = M - \alpha \frac{L}{K}$, where $\alpha >0$. Moreover, let $k$ be such that $a \in B_k$. Then $$\min_{a' \in B_k} m(a') \geq M - \left(\alpha + (\alpha \lor \sqrt{d})\right)\frac{L}{K},$$ and $$\max_{a' \in B_k} m(a') \leq M - \left(\alpha - \frac{(\alpha \lor 2\sqrt{d})}{2}\right)\frac{L}{K}.$$
\end{lem}

\begin{proof}
In the general $d$-dimensional case, two points in the same bin may be separated by a Euclidean distance of $\sqrt{d}/K$. Using this remark, one can readily adapt the proof of Lemma \ref{lem:lipschitz} to prove Lemma \ref{lem:d-lipschitz}.
\end{proof}

Conversely, we obtain a lower bound on the Lebesgue measure of arms with mean reward close to $M$. 

\begin{lem}\label{lem:isoperimetrie}
There exist a constant $c_{p,d}>0$ depending only on $p$ and $d$ such that for all $t\in (0, \sqrt{d}L]$,
\begin{align*}
\mathbb{P}\left(m(a_1) \in [M, M+t)\right) \geq c_{p,d}\frac{t}{L}.
\end{align*}
\end{lem}

Then, we decompose the regret $R_T$ into the sum of a discretization error, and of the cost of learning in the corresponding finite $K^d$-armed bandit problem. For $k = 0, ..., K^d-1$, we  define $m_k = K^d\int_{a \in B_K}m(a)da$ as the mean payment for pulling an arm uniformly in bin $B_k$. In order to avoid cumbersome notations for reordering the bins, we assume henceforth (without loss of generality) that $\left\{m_k\right\}_{0\leq k \leq K^d-1}$ is a decreasing sequence. Similarly to the one-dimensional case, we denote by $\phi^d$ the strategy pulling all arms in the bin $B_1$, $B_2$, up to $B_{\widehat{f}}$ and pulling the remaining arms in $B_{\hat{f}+1}$, where $\widehat{f}$ is such that $ N_{1} + .. +  N_{\widehat{f}}< T \leq N_{1} + .. +  N_{\widehat{f}+1}$. Note that $\phi^d$ corresponds to the oracle strategy for the discretized problem. We also denote $f = \lfloor pK^d\rfloor$. Recall that we denote by $\phi^*(t)$ the arm pulled at time $t$ by the oracle strategy, and by $\phi(t)$ the arm pulled at time $t$ by UCBF.

Now, decompose $R_T$ as follows :
\begin{eqnarray*}
    R_T &=& \underset{t = 1.. T }{\sum} m(a_{\phi^*(t)}) - \underset{t = 1..T}{\sum} m(a_{\phi(t)})\\
     &=& \underset{t = 1.. T }{\sum} m(a_{\phi^*(t)}) - \underset{t = 1..T}{\sum} m(a_{\phi^d(t)}) + \underset{t = 1.. T }{\sum} m(a_{\phi^d(t)}) - \underset{t = 1..T}{\sum} m(a_{\phi(t)}).
\end{eqnarray*}

Again, we denote by $R_T^{(d)} = \underset{t = 1.. T }{\sum} m(a_{\phi^*(t)}) - \underset{t = 1..T}{\sum} m(a_{\phi^d(t)})$ the discretization error. Moreover, we define $R_T^{(FMAB)} = \underset{t = 1.. T }{\sum} m(a_{\phi^d(t)}) - \underset{t = 1..T}{\sum} m(a_{\phi(t)})$ the regret of our strategy against the oracle strategy for the discrete problem.

As in the one-dimensional case, we use the following Lemmas to bound the discretization error $R_T^{(d)}$. 

\begin{lem}\label{lem:d-control_hat_f}
Define $\epsilon = \lceil c_{p,d} K^{d-1}\rceil$ and $\alpha = 4QL/c_{p,d} + 2/\sqrt{d} \times (1+3/K^{d-1})$, where $c_{p,d}$ is the constant appearing in Lemma \ref{lem:isoperimetrie}. With probability at least $1 - 4\exp\left(-\frac{2c_{p,d}^2N}{K^{2}}\right)$, we have $\vert\widehat{f} - f \vert \leq 1+ \epsilon$. On this event, $\left \vert m_{\hat{f}} - M \right \vert \leq \alpha \sqrt{d}L/K$ and $\left \vert m_{\hat{f}+1} - M \right \vert \leq \alpha \sqrt{d}L/K$.
\end{lem}

\begin{lem}\label{lem:d-control_hat_M}
For the constant $c_{p,d}>0$ defined in Lemma \ref{lem:isoperimetrie},
\begin{equation}\label{eq:bound_M_d_p}
    \mathbb{P}\left(\vert \widehat{M} - M \vert \leq \sqrt{d}L/K\right) \geq 1  - 2e^{-\frac{2c_{p,d}^2N}{K^2}}.
\end{equation}
\end{lem}

The proof of Lemma \ref{lem:d-control_hat_M} is obtained by following the lines of the proof of Lemma \ref{lem:control_hat_M}, and applying Lemma \ref{lem:isoperimetrie}. It is therefore omitted.

\begin{lem}\label{lem:d-ech_arms} With probability at least $1 - 2\exp\left(-\frac{8\alpha^2dL^2Q^2N}{K^2}\right)$, 
$$\left \vert \left\{i : m(a_i) \in \left[M - \frac{2\alpha\sqrt{d}L}{K}, M+\frac{ L\sqrt{d}}{K}\right]\right\}\right\vert \leq \frac{4\alpha \sqrt{d}LQ N}{K}.$$
\end{lem}

The proof of Lemma \ref{lem:d-ech_arms} follows from the arguments developed in the proof of Lemma \ref{lem:ech_arms}, and is therefore omitted. Note that since $LQ \geq 1$, $d\geq 1$ and $\alpha \geq1 \geq c_{p,d}$, $1 - 2\exp\left(-\frac{8\alpha^2dL^2Q^2N}{K^2}\right) \geq 1 - 2\exp\left(-\frac{2c_{p,d}^2N}{K^2}\right)$.

\begin{lem}\label{lem:d-E_a}
Let
\begin{eqnarray*}\label{lem:controle_R1}
    \mathcal{E}_{a} & =& \left\{ \vert\widehat{f} - f \vert \leq 1 + \epsilon \right\}\cap \left\{\vert \widehat{M} - M \vert \leq \sqrt{d}L/K \right\} \\
    &\cap& \left\{ \left \vert \left\{i : m(a_i) \in \left[M - \frac{2\alpha\sqrt{d}L}{K}, M+\frac{\sqrt{d}L}{K}\right]\right\}\right\vert \leq \frac{4\alpha\sqrt{d}LQ N}{K}\right\}.
\end{eqnarray*}
On the event $\mathcal{E}_{a}$, $R_T^{(d)} \leq \frac{8\alpha^2dQL^2N}{K^2}$.
\end{lem}

Combing Lemmas \ref{lem:d-control_hat_f}, \ref{lem:d-control_hat_M} and \ref{lem:d-ech_arms}, we note that $\mathbb{P}(\mathcal{E}_a) \geq 1 - 8 \exp(\frac{2c_{p,d}^2N}{K^2})$. Next, we bound the cost of learning on the corresponding finite $K^d$-armed bandit problem. Similarly to the one-dimensional case, we use the following Lemmas to control this term.

\begin{lem}\label{lem:d-equi_nk}
\begin{eqnarray*}
\mathbb{P}\left(\underset{k \in \{0,..,K^d-1\}}{\max} \left \vert N_k - \frac{N}{K^d} \right \vert \geq \frac{N}{2K^d}\right) \leq   2K^de^{-\frac{N}{10K^d}}.
\end{eqnarray*}
\end{lem}

\begin{lem}\label{lem:d-borne_m} 
\begin{eqnarray*}
\mathbb{P}\left( \exists k \in \{0,...,K^d-1\}, s \leq (N_k\land T) : \left\vert \widehat{m}_k(s) - m_k \right\vert \geq \sqrt{\frac{\log(T/\delta)}{2s}}\right) &\leq& 2K^d\delta.
\end{eqnarray*}
\end{lem}

The proof of Lemma \ref{lem:d-borne_m} follows closely the proof of Lemma \ref{lem:borne_m}, and is therefore omitted.

Now, we define
\begin{eqnarray*}
\mathcal{E}_b = && \left\{\underset{k = 0,..,K^d-1}{\cap} \left\{N_k \in \left[\frac{N}{2K^d}, \frac{3N}{2K^d}\right]\right\}\right\}\\
&& \cap \left\{\underset{k = 0,..,K^d-1}{\cap} \underset{s = 1..(N_k\land T)}{\cap} \left\{ \vert m_k - \widehat{m}_k(s)\vert \leq \sqrt{\frac{\log(T/\delta)}{2s}} \right\}\right\}.
\end{eqnarray*}

Combining Lemma \ref{lem:d-equi_nk} and Lemma \ref{lem:borne_m}, we find that $$\mathbb{P}\left( \mathcal{E}_b\right) \geq 1 - 2K^de^{-\frac{N}{10K^d}} - 2K^d\delta.$$
For two bins $k,l \in \{1,...,K\}^2$, we  denote henceforth $\Delta_{k,l} = m_k - m_l$.
\begin{lem}\label{lem:d-bound_nl}
Let $k\in \{1,...,K\}$. On the event $\mathcal{E}_b \cap \{n_k(T) < N_k \}$, a.s. for all bins $B_l$ such that $\Delta_{k,l} > 0$, $n_l(T) \leq \frac{3\log(T/\delta)}{\Delta_{k,l}^2}$.
\end{lem}

The proof of Lemma \ref{lem:d-bound_nl} can be obtained by following the lines of the proof of Lemma \ref{lem:equi_nk}, and is therefore omitted. 

As in the one-dimensional case, we write $R_T^{(FMAB)} = R_{opt} + R_{\hat{f}+1} + R_{subopt}$, where 
$$R_{\hat{f}+1} = \underset{\ a_i \in B_{\hat{f}+1} \cap \Phi^d(T) \cap \overline{\Phi(T)}}{\sum} \left(m(a_{i}) - M\right) \nonumber  +\underset{\ a_i \in B_{\hat{f}+1} \cap \Phi(T) \cap \overline{\Phi^d(T)}}{\sum} \left(M-m(a_{i}) \right),$$
$$R_{opt} = \underset{k = 1..\widehat{f}}{\sum}\  \underset{a_i \in B_k \cap \overline{\Phi(T)}}{\sum} \left(m(a_{i}) - M\right),$$
and $$R_{subopt} = \underset{k = \widehat{f}+2 .. K^d-1}{\sum}\ \underset{a_i \in B_k \cap \Phi(T)}{\sum} \left(M  - m(a_i)\right).$$

The term $R_{\hat{f}+1}$ can easily be bounded : there are most $1.5N/K^d$ arms in $B_{\hat{f}+1}$, and so there are at most $1.5N/K^d$ terms in $R_{\hat{f}+1}$. On the event $\mathcal{E}_a$, $m_{\hat{f}+1} \in [M-\alpha \sqrt{d}L/K, M+\alpha \sqrt{d}L/K]$. Using Lemma \ref{lem:lipschitz}, we see that for each arm $a_i \in B_{\hat{f}+1}$, $\vert m(a_i) - M\vert\leq 2\alpha \sqrt{d}L/K$. Thus, on $\mathcal{E}_a \cap \mathcal{E}_b$, $R_{\hat{f}+1} \leq 3\alpha \sqrt{d}LN/K^{d+1}$.

The following Lemmas help us bound the terms $R_{subopt}$ and $R_{opt}$.

\begin{lem}\label{lem:d-R_subopt}
On the event $\mathcal{E}_a \cap \mathcal{E}_b$, 
\begin{equation*}
    R_{subopt} \leq  120\alpha^2dL^2Q\left(\frac{N}{K^2} + \frac{K^{d}\log(T/\delta)\log_2(K/\alpha\sqrt{d}L)}{\alpha^2dL^2}\right).
\end{equation*}
\end{lem}

\begin{lem}\label{lem:d-n_subopt}
On $\mathcal{E}_a \cap \mathcal{E}_b$, the number of arms pulled in sub-optimal bins by UCBF is bounded by $6\alpha\sqrt{d}LQN/K + 24 \log(T/\delta)K^{d+1}Q/(\alpha\sqrt{d}L)$.
\end{lem}

\begin{lem}\label{lem:d-used_up}
Let $$A =  \sqrt{\frac{472QK^{d+2}\log(T/\delta)}{Nc_{p,d}Ld}} \lor 16\alpha QL/c_{p,d}.$$ At time $T$, on the event $\mathcal{E}_{a}\cap\mathcal{E}_{b}$, all bins $B_k$ such that $m_k \geq M + A\sqrt{d}L/K$ have died.
\end{lem}

Combining Lemmas \ref{lem:d-R_subopt}, \ref{lem:d-n_subopt} and \ref{lem:d-used_up}, we prove the following result.

\begin{lem}\label{lem:d-bound_Ropt}
On event $\mathcal{E}_a \cap \mathcal{E}_b$, 
\begin{equation*}
R_{opt}\leq 30\alpha A d L^2Q\left(\frac{N}{K^2} + \frac{\log(T/\delta)K^{d}}{\alpha^2 dL^2}\right).
\end{equation*}
\end{lem}
The proof of Lemma \ref{lem:d-bound_Ropt} is similar to that of Lemma \ref{lem:bound_Ropt}, and is therefore omitted.

Thus, on the event $\mathcal{E}_a \cap \mathcal{E}_b$,
\begin{eqnarray*}
R_T &\leq & \left(33\alpha A d L^2Q +128\alpha^2dL^2Q\right)\left(\frac{N}{K^2} + \frac{\log(T/\delta)\log_2(K/\alpha\sqrt{d}L)K^{d}}{\alpha^2 dL^2}\right).
\end{eqnarray*}

The event $\mathcal{E}_a\cap \mathcal{E}_b$ happens with probability larger than $1 - 8 \exp(\frac{2c_{p,d}^2N}{K^2}) - 2K^d\exp(-\frac{N}{10K^d}) - 2K^d\delta$. For the choice $K = \lceil N^{\frac{1}{d+2}}\log(N)^{-\frac{2}{d+2}}\rceil$ and $\delta = N^{-\frac{2d+2}{d+2}}$,
\begin{eqnarray*}
\mathbb{P}\left(\mathcal{E}_a \cap \mathcal{E}_b\right) &\geq& 1 - 8\exp\left(-2c_{p,d}^2N^{\frac{d}{d+2}}\log(N)^{\frac{4}{d+2}}\right) - 2(N^{\frac{1}{d+2}}+1)^{d}\log(N)^{\frac{-2d}{d+2}}\exp\left(-N^{\frac{2}{d+2}}\log(N)^{\frac{2d}{d+2}}/10\right)  \\
&& + 2(N^{\frac{1}{d+2}}\log(N)^{\frac{-2}{d+2}}+1)^dN^{\frac{-(2d+2)}{d+2}}\\
&\geq & 1 - O(N^{-1}).
\end{eqnarray*}

Note that for this choice of $K$, $A$ is bounded by a constant depending on $\alpha$, $Q$, $L$ and $c_{p,d}$. Then, $\mathcal{E}_a \cap \mathcal{E}_b$, there exists a constant $C$ depending on $d$, $L$, $Q$ and $p$ such that 
\begin{eqnarray*}
R_T &\leq & C \left(N^{\frac{d}{d+2}}\log(N)^{\frac{4}{d+2}} +  \frac{\frac{3+2d}{(d+2)^2}\log(N)\log(N)(N^{\frac{1}{d+2}}\log(N)^{\frac{-2}{d+2}}+1)^d}{\alpha^2 dL^2}\right).
\end{eqnarray*}

This concludes the proof of Theorem \ref{thm:d-bound_high_prob}.


\section{Proofs of auxiliary Lemmas}
\label{subsec:auxi}

\subsection{Proof of Lemma \ref{lem:lipschitz}}

Recall that $a\in I_k$ and $\alpha >0$ is such that $m(a) = M + \alpha L/K$. By Assumption \ref{hyp:lip}, we see that for any $a' \in I_k$,
\begin{eqnarray*}
\vert (M + \alpha L/K) - m(a') \vert &\leq & \max \{ \alpha L/K, L/K \},
\end{eqnarray*}
so
\begin{eqnarray*}
m(a') &\leq & M + (\alpha + (\alpha \lor 1))L/K.
\end{eqnarray*}
This yield the first part of the Lemma. To obtain the second part, note that Assumption \ref{hyp:lip} also implies
\begin{eqnarray}
\vert m(a') - (M + \alpha L/K) \vert &\leq & \max \{ \vert m(a') - M \vert , L/K \}\nonumber.
\end{eqnarray}
Thus,
\begin{eqnarray}
m(a') &\geq & M + \alpha L/K  -  \max \{ \vert m(a') - M \vert , L/K \}.\label{eq:cas_1}
\end{eqnarray}
If $\vert m(a') - M \vert \geq L/K$, then equation \eqref{eq:cas_1} implies
\begin{eqnarray*}
m(a') &\geq & M + \alpha L/K  -  (m(a') - M).
\end{eqnarray*}
Thus,
\begin{eqnarray*}
2m(a') &\geq& 2M + \alpha L/K
\end{eqnarray*}
and 
\begin{eqnarray*}
m(a')&\geq& M + \frac{\alpha}{2} L/K.
\end{eqnarray*}
Since $\vert m(a') - M \vert =  m(a') - M = \alpha L/(2K)$, $\vert m(a') - M \vert \geq L/K$ implies $\alpha \geq 2$. On the other hand, if $\vert m(a') - M \vert < L/K$, equation \ref{eq:cas_1} implies
\begin{eqnarray*}
m(a') &\geq & M + \alpha L/K  -  L/K\\
&\geq& M + \frac{(\alpha-1)L}{K}.
\end{eqnarray*}
Since $m(a') - M \leq \vert m(a') - M \vert$, the assumption $\vert m(a') - M \vert < L/K$ implies that $\alpha < 2$.

To summarise, when $\alpha < 2$ we necessarily have $\vert m(a') - M \vert < L/K$, and $m(a') \geq M + (\alpha-1)L/K$. On the contrary,  when $\alpha \geq 2$ we necessarily have $\vert m(a') - M \vert \geq L/K$, and $m(a') \geq M + \alpha L/(2K)$. This writes
\begin{eqnarray*}
m(a') &\geq& M + \left(\alpha - \frac{\alpha\lor 2}{2}\right)\frac{L}{K}.
\end{eqnarray*}
Using the same arguments, we can prove similar bounds for the case $m(a) = M - \alpha L/K$.
\subsection{Proof of Lemma \ref{lem:control_hat_f}}
Recall that $f = \lfloor pK\rfloor$, and $\widehat{f}$ is such that $ N_{1} + .. +  N_{\widehat{f}}< T \leq N_{1} + .. +  N_{\widehat{f}+1}$. By definition, $N_{1}  + .. + N_{f-1} = \underset{1\leq i \leq N}{\sum} \mathds{1}_{\{a_i \in I_{1}\cup .. \cup I_{f-1}\}}$, where $\mathds{1}_{\{a_i \in I_{1}\cup .. \cup I_{f-1}\}}$ are independant Bernoulli random variables of parameter $\frac{f-1}{K}$. Using Hoeffding's inequality, we find that

\begin{eqnarray*}
\mathbb{P}\left( \underset{1\leq i \leq N}{\sum} \mathds{1}_{\{a_i \in I_{1}\cup .. \cup I_{f-1}\}} - \frac{(f-1)N}{K} \geq \frac{N}{K} \right) &\leq& e^{-\frac{2N}{K^2}}.
\end{eqnarray*}
Now, by definition, $f = \lfloor TK/N \rfloor$, and so $fN/K \leq T$. Thus, 
\begin{eqnarray}\label{eq:f-1}
\mathbb{P}\left( N_{1} + .. +N_{f-1}  \geq T \right) &\leq& e^{-\frac{2N}{K^2}}.
\end{eqnarray}

This shows that with high probability, $N_{1} + .. +N_{f-1}<T$, which implies that $f-1 < \hat{f}+1$. Using again Hoeffding's inequality, we find that

\begin{eqnarray*}
\mathbb{P}\left(  \frac{(f+2)N}{K} -\underset{1\leq i \leq N}{\sum} \mathds{1}_{\{a_i \in I_{1}\cup .. \cup I_{f+2}\}} \geq \frac{N}{K} \right) &\leq& e^{-\frac{2N}{K^2}}.
\end{eqnarray*}

By definition of $f$, $(f+1)N/K \geq T$. Thus, 
\begin{eqnarray}\label{eq:f+2}
\mathbb{P}\left( N_{1} + .. +N_{f+2}  \geq T \right) &\leq& e^{-\frac{2N}{K^2}}
\end{eqnarray}

This shows that with high probability, $T<N_{1} + .. +N_{f+2}$, and thus $f+2 > \hat{f}$. Combining  equations \eqref{eq:f-1} and \eqref{eq:f+2}, we find that with probability larger than $1- 2e^{-\frac{2N}{K^2}}$, $\vert f - \hat{f}\vert \leq 1$.

\medskip

In a second time, we prove that $m_{f} \in [M-L/K, M+L/K]$. To do so, we first show that there are at least $\lceil pK \rceil$ intervals $k$ such that $m_k \geq M - L/K$, or equivalently that there are at most $\lfloor (1-p)K \rfloor $ intervals $k$ such that $m_k < M - L/K$. Indeed, for all $k$ such  that $m_k < M - L/K$, there exists $a \in I_k$ such that $m(a) < M-L/K$. Using Lemma \ref{lem:lipschitz}, we see that $\forall a \in I_k$, $m(a) \leq M$. By definition of $p$, there can be at most $\lfloor (1-p)K \rfloor$ such intervals. Therefore, there are at least $\lceil pK \rceil$ intervals $k$ such that $m_k \geq M - L/K$. Since $f < \lfloor pK\rfloor$, this implies that $m_{f} \geq M - L/K$. Similar arguments show that $m_{f} \leq M + L/K$.

We conclude by noting that since $m_f \geq M - L/K$, Lemma \ref{lem:lipschitz} implies  $\min_{a\in \cup_{k\leq f}I_k } m(a)\geq M-2L/K$. We define $\tilde{a} = \argmax\{m(a): a \in \cup_{k>f} \overline{I_k}\}$. The continuity of $m$ implies that $m(\tilde{a}) \geq M - 2L/K$. Let $\tilde{k}>f$ be such that $\tilde{a}\in \overline{I_{\tilde{k}}}$. Then, Lemma \ref{lem:lipschitz} implies that $m_{\tilde{k}} \geq M - 4L/K$. Since $m_{f+1} = \max_{k>f} m_k$, this implies in particular $m_{f+1} \geq M-4L/K$. Similar arguments can be used to show that $m_{f+2} \geq M-8L/K$ and that $m_{f-1} \leq M+4L/K$. Thus, when $\vert\hat{f} - f \vert \leq 1$, we find that $m_{\widehat{f}} \in [M - 4L/K, M + 4L/K]$, and $m_{\widehat{f}+1} \in [M - 8L/K, M + L/K]$.
\subsection{Proof of Lemma \ref{lem:control_hat_M}}

Recall that $\widehat{M} = m(a_{\phi^*(T)})$, where $T= pN$ and $\phi^*$ is a permutation such  that $\{m(a_{\phi^*(i)})\}_{1\leq i \leq N}$ is a decreasing sequence. Thus,  $\widehat{M}$ is the $T$-th largest payment for the arms  with covariates $\{a_1, ..., a_N\}$. To bound its deviation from its expected value $M$, we note that for all $t>0$, $\left\{\widehat{M} \geq M + t\right\}$ implies $\left\{ \underset{1\leq i \leq N}{\sum} \mathds{1}_{\{m(a_i)\geq M+t\}} \geq T\right\}$. Since $T = Np = N \mathbb{P}\left(m(a_1) \geq M\right)$,
\begin{align*}
\mathbb{P}&\left(\widehat{M} \geq M + t \right) \leq \mathbb{P}\left(\underset{1\leq i \leq N}{\sum} \mathds{1}_{\{m(a_i)\geq M+t\}} \geq N\mathbb{P}\left(m(a_1) \geq M\right)\right)\\
&\leq \mathbb{P}\left(\underset{1\leq i \leq N}{\sum} \left(\mathds{1}_{\{m(a_i) \geq M+t\}} - \mathbb{P}\left(m(a_1) \geq M+t \right) \right) \geq N\mathbb{P}\left(m(a_1) \in [M, M+t)\right)\right).
\end{align*}

Using Hoeffding's equality, we find that
\begin{align*}
\mathbb{P}\left(\widehat{M} \geq M + t \right) \leq \exp\left(-2N\mathbb{P}\left(m(a_1) \in [M, M+t)\right)^2 \right).
\end{align*}

For the choice $t = L/K$, it implies that
\begin{align*}
\mathbb{P}\left(\widehat{M} \geq M + L/K \right) \leq \exp\left(-2N\mathbb{P}\left(m(a_1) \in [M, M+L/K)\right)^2 \right).
\end{align*}

Next, we obtain a lower bound on $\mathbb{P}\left(m(a_1) \in [M, M+L/K)\right)$. Note that either $\max \{m(a) : a\in[0,1] \} \leq M+L/K$, and $\mathbb{P}\left(m(a_1) \in [M, M+L/K)\right) = \mathbb{P}\left(m(a_1) \geq M\right) = p \geq 1/K$, or $\max \{m(a) : a\in[0,1] \} > M+L/K$. 

In this case, choose $a^{(1)} \in \argmax_a\{m(a)\}$ ($a^{(1)} $ exists since $m$ is continuous and defined on a compact set). Note that $m(a^{(1)}) > M+L/K$. Since $m$ is continuous and $\lambda(\{a : m(a)<M\}) >0$ (because of Assumption \ref{hyp:beta} and the fact that $p<1$), $\{a : m(a) = M\}\neq \emptyset$. Define $a^{(2)} = \argmin_a\{\vert a - a^{(1)}\vert : m(a) = M\}$, and assume without loss of generality that $a^{(1)} \leq a^{(2)}$. Since $m$ is continuous, $m(a^{(1)}) > M+L/K$ and $m(a^{(2)}) = M$, we see that $\{a \in [ a^{(1)}, a^{(2)}]: m(a) = M +L/K\}\neq \emptyset$. Define finally $a^{(3)} = \max\{a : a \leq a^{(2)}, m(a) = M+L/K\}$. By construction, for all $a \in [a^{(3)}, a^{(2)})$, $m(a) \in [M, M+L/K)$. Using Assumption \ref{hyp:lip}, we find that $\vert a^{(3)}- a^{(2)}\vert \geq 1/K$. Thus, $\mathbb{P}\left(m(a_1) \in [M, M+L/K\right) \geq \mathbb{P}\left(m(a_1) \in [a^{(3)}, a^{(2)}] \right) \geq 1/K$.

Putting things together, we find that 
\begin{eqnarray*}
\mathbb{P}\left(\widehat{M} \geq M + L/K \right) &\leq& \exp\left(-2N/K^2\right).
\end{eqnarray*}
Using similar arguments, we can show that $\mathbb{P}\left(\widehat{M} \leq M - L/K \right) \leq \exp\left(-2N/K^2\right)$.

\subsection{Proof of Lemma \ref{lem:ech_arms}}

In order to prove Lemma \ref{lem:ech_arms}, we first state the following result.

\begin{lem}\label{lem:arms_in_mu}
Let $\mathcal{B}$ be a Borel set of measure $\lambda(\mathcal{B}) \geq N^{-2/3}$, and $N_{\mathcal{B}}$ be the number of arms in $\mathcal{B}$. Then, 
\begin{eqnarray*}
\mathbb{P}\left( \vert N_{\mathcal{B}} - \lambda(\mathcal{B}) N \vert \geq \sqrt{\lambda(\mathcal{B}) N^{4/3}}\right) \leq  2e^{-\frac{N^{1/3}}{3}}.
\end{eqnarray*}
\end{lem}
\begin{proof}
Recall that $N_{\mathcal{B}} = \underset{1 \leq i \leq N}{\sum}\mathds{1}_{a_i \in \mathcal{B}}$, where $\mathds{1}_{a_i \in \mathcal{B}} \overset{i.i.d}{\sim}$Bernoulli$(\lambda(\mathcal{B}))$. Applying Bernstein's inequality, we find that
\begin{eqnarray*}
\mathbb{P}\left(\vert N_{\mathcal{B}}- \lambda(\mathcal{B}) N \vert \geq t\right) &\leq& 2e^{-\frac{t^2}{2\lambda(\mathcal{B}) N + 2t/3}}\\
\mathbb{P}\left(\vert N_{\mathcal{B}} - \lambda(\mathcal{B}) N \vert \geq  \sqrt{\lambda(\mathcal{B}) N^{4/3}} \right) &\leq& 2e^{-\frac{N^{1/3}}{3}}
\end{eqnarray*}
\end{proof}
Now, we use Lemma \ref{lem:arms_in_mu} for $\mathcal{B} = \left\{x : m(x) \in [M-16L/K, M+L/K]\right\}$. By Assumption \ref{hyp:beta}, $\lambda\left(\left\{x : m(x) \in [M-16L/K, M+L/K]\right\}\right) \leq 16LQ/K$. When $K \leq N^{2/3}$, the inequality $QL\geq 1$ implies that $K \leq 16LQN^{2/3}$, and $\sqrt{16LQN^{4/3}/K} \leq 16 LQN/K$. This proves Lemma \ref{lem:ech_arms}.

\subsection{Proof of Lemma \ref{lem:E_a}}
Non-zero terms in $R_T^{(d)}$ correspond to pairs of arms $(i, j)$ such that $i$ is pulled by $\phi^d$ but not by $\phi^*$, and $j$ is pulled by $\phi^*$ but not by $\phi^d$. If an arm $i$ is pulled by $\phi^d$, it belongs to an interval $k$ such that $m_k \geq m_{\hat{f}+1}$. On the event $\mathcal{E}_a$, $$m_{\hat{f}+1} \geq M - 8L/K.$$ Using Lemma \ref{lem:lipschitz}, we find that $$m(a_i) \geq M - 16L/K.$$ On the other hand, if $i$ is not pulled by $\phi^*$, it must be such that $m(a_i) \leq \widehat{M}$. On the event $\mathcal{E}_a$, this implies that $m(a_i) \leq M + L/K$. Since there are at most $\frac{32LQ N}{K}$ arms in $[M-16L/K, M + L/K]$ on the event $\mathcal{E}_a$, there are at most $\frac{32LQ N}{K}$ arms that are selected by $\phi^d$ and not by $\phi^*$, and thus at most $\frac{32LQ N}{K}$ non-zero terms in $R_T^{(d)}$.

Now, each of these terms corresponds to the cost of pulling an arm $i$ selected by $\phi^d$ but not by $\phi^*$, instead of an arm $j$ selected by $\phi^*$ but not by $\phi^d$. Assume that $m_{\hat{f}+1}\geq M$. Then, using Lemma \ref{lem:lipschitz}, we see that if $i$ is selected by $\phi^d$, $m(a_i) \geq M - L/K$. Moreover, if $j$ is not selected by $\phi^d$, it belongs to an interval $I_k$ such that $m_k \leq  m_{\hat{f}+1}$. On $\mathcal{E}_a$, $m_{f+1}\leq M + L/K$. Thus, $m(a_j) \leq M + 2L/K$, and $m(a_j) - m(a_i) \leq 3L/K$. On the other hand, if $m_{\hat{f}+1} < M$, then according to Lemma \ref{lem:lipschitz} for all $i$ selected by $\phi^d$, $m(a_i) \geq M - 2\left((M-m_{\hat{f}+1})\lor L/K\right)$, while for $j$ not selected by $\phi^d$, $m(a_j) \leq m_{\hat{f}+1} + (M - m_{\hat{f}+1})/2\lor L/K$. Thus, $m(a_j) - m(a_i) \leq 3/2\left((M - m_{\hat{f}+1})\lor 2L/K \right)\leq 12L/K$.

To conclude, on the event $\mathcal{E}_a$ there are at most $\frac{32LQN}{K}$ non-zero terms in $R_T^{(d)}$, and each of them is bounded by $12L/K$. Thus, 
$$R_T^{(d)} \leq \frac{32QLN}{K} \times 12L/K.$$

\subsection{Proof of Lemma \ref{lem:equi_nk}}

Note that for $k \in \{1,...,K\}$, $I_k$ is a Borel set of measure $1/K \geq N^{-2/3}$. Using Lemma \ref{lem:arms_in_mu}, we find that
\begin{eqnarray*}
\mathbb{P}\left(\left \vert N_k - \frac{N}{K}\right \vert \geq \frac{N^{2/3}}{K^{1/2}} \right) &\leq& 2e^{-\frac{N^{1/3}}{3}}\\
\mathbb{P}\left(\left \vert N_k - \frac{N}{K}\right \vert \geq \frac{N}{2K}\times\frac{2K^{1/2}}{N^{1/3}} \right) &\leq& 2e^{-\frac{N^{1/3}}{3}}.
\end{eqnarray*}
Since $K \leq N^{2/3}/4$, $2K^{1/2}/N^{1/3}\leq 1$.
A union bound for $k = 1, ..., K$ yields the result.

\subsection{Proof of Lemma \ref{lem:borne_m}}

Recall that $a_i \sim \mathcal{U}([0,1])$, and thus $a_i \big \vert \{a_i \in I_k\} \sim \mathcal{U}(I_k)$. Since the arms $a_{\pi_k(s)}$ are selected uniformly at random among the arms in $I_k$, they are independent from one another, and uniformly distributed on $I_k$. 

For $k \in 1, ..., K$ and for $n \in [0, N]$, we denote by $\mathbb{P}_{n}$ the probability measure obtained by conditioning on the event $N_k = n$ (this event has a strictly positive probability because $\lambda(I_k) \in (0,1)$). Note that for any $s \in [1, n]$, $\mathbb{E}_{n}[y_{\pi_k(s)}] = m_k$. Using Hoeffding's inequality, we find that for any $n \in [1, N]$ and any $s \in [1, n]$

\begin{eqnarray*}
\mathbb{P}_{n}\left( \left\vert \frac{1}{s} \underset{1 \leq i \leq s}{\sum} y_{\pi_k(i)} - m_k \right\vert \geq \sqrt{\frac{\log(T/\delta)}{2s}} \right) \leq 2e^{-\log(T/\delta)} = \frac{2\delta}{T}
\end{eqnarray*}

The inequality $\vert \widehat{m}_k(0) - m_k\vert \leq \infty$ also holds, since we defined $\widehat{m}_k(0) = 0$, and thus the inequality above is also verified for $n = 0$. Using a union bound for $s = 0, ..., (n \land T)$, we find that for all $n = 0, ..., N$,
\begin{eqnarray*}
\mathbb{P}_{n}\left(\exists s \leq (n \land T) : \left \vert  \widehat{m}_k(s) - m_k \right\vert \geq \sqrt{\frac{\log(T/\delta)}{2s}} \right) &\leq& \frac{2\delta (n\land T)}{T}\leq 2\delta.
\end{eqnarray*}

We integrate over the different values of $n$ and find that 

\begin{eqnarray*}
\mathbb{P}\left(\exists s \leq (N_k\land T) : \left\vert \widehat{m}_k(s) - m_k \right\vert \geq \sqrt{\frac{\log(T/\delta)}{2s}} \right) &\leq& 2\delta.
\end{eqnarray*}

Finally, a union bound for $k = 1, ..., K$ yields
\begin{eqnarray*}
\mathbb{P}\left( \exists k \in \{1,...,K\}, s \leq (N_k \land T) : \left\vert \widehat{m}_k(s) - m_k \right\vert \geq \sqrt{\frac{\log(T/\delta)}{2s}}\right) &\leq& 2K\delta.
\end{eqnarray*}

\subsection{Proof of Lemma \ref{lem:bound_nl}}
First, note that on $\mathcal{E}_b$, all intervals are non-empty. By definition of Algorithm UCBF, at least one arm is pulled in each interval. To bound the number of arms pulled in interval $I_l$, assume that time $t>K$ is such that the arm $\phi(t)$ is selected in $I_l$. Since there are arms available in $I_k$ at time $T$, there are arms available in $I_k$ at time $t\leq T$. If UCBF pulls an arm in $I_l$ instead of an arm in $I_k$, we must have
\begin{eqnarray*}
\widehat{m}_{k}(n_{k}(t-1)) + \sqrt{\frac{\log(T/\delta)}{2n_{k}(t-1)}} \leq \widehat{m}_{l}(n_{l}(t-1)) + \sqrt{\frac{\log(T/\delta)}{2n_{l}(t-1)}}.
\end{eqnarray*}
On the event $\mathcal{E}_b$, this implies that 
\begin{eqnarray*}
m_{k}&\leq& m_{l}+ 2\sqrt{\frac{\log(T/\delta)}{2n_{l}(t-1)}}.
\end{eqnarray*}
Straightforward calculations show that
\begin{eqnarray*}
n_{l}(t-1) &\leq& \frac{2\log(T/\delta)}{\Delta_{k,l}^2}.
\end{eqnarray*}
Thus $n_{l}(T) \leq \left(\frac{2\log(T/\delta)}{\Delta_{k,l}^2} \lor 1\right) + 1 \leq \frac{3\log(T/\delta)}{\Delta_{k,l}^2}$ since $\Delta_{k,l}^2\leq 1$ and $\log(T/\delta)\geq 1$.


\subsection{Proof of Lemma \ref{lem:R_subopt}}

By Lemma \ref{lem:control_hat_f}, on the event $\mathcal{E}_a$, $m_{\hat{f}+1} \in [M-8L/K, M+L/K]$. We group intervals with mean rewards lower than $m_{\widehat{f}+1}$ into the following subsets.

Let $\mathcal{S}_0 = \left\{ k : (M - m_{k}) \in [-L/K, 10L/K] \right\}$, $\mathcal{S}_1 = \left\{ k : (M - m_{k}) \in (10L/K, 16L/K] \right\}$, and for $n\geq 2$ define $\mathcal{S}_n = \left\{ k : (M - m_{k}) \in [2^{n+2}L/K, 2^{n+3}L/K \right\}$. Note that for $n \geq \log_2(K/L)  - 2$, $\mathcal{S}_n$ is empty since $m$ is bounded by $1$. 

Using Lemma \ref{lem:lipschitz}, we note that for all $l \in \mathcal{S}_0 $ and all $a \in I_l$, $$\vert m(a) - M \vert \leq 20L/K.$$ Using Assumption \ref{hyp:beta}, we conclude that $\left \vert \mathcal{S}_0 \right \vert \leq 20LQ$. On $\mathcal{E}_a$, there are at most $1.5N/K$ arms in each interval, so the number of arms in intervals in $\mathcal{S}_0$ is at most $30LQN/K$. Moreover for all $l \in \mathcal{S}_0$ and all $a_i \in I_l$, $(M - m(a_i))\leq 20L/K$. Thus, the arms pulled from intervals in $\mathcal{S}_0$ contributes to $R_{subopt}$ by at most $20L/K \times 30LQN/K = 600L^2QN/K^2$.

Similarly, for all $l \in \mathcal{S}_1$ and all $a \in I_l$, $$\vert m(a) - M \vert \leq 32L/K.$$ Using Assumption \ref{hyp:beta}, we conclude that $\left \vert \mathcal{S}_1 \right \vert \leq 32LQ$. Moreover, by definition of $\widehat{f}$, there exists a interval $I_k$ with $m_k \geq m_{\widehat{f}+1}$ such that $n_k(T) < N_k$. Since $\Delta_{k,l} \geq m_{\hat{f}+1} - m_l \geq M - 8L/K - (M - 10L/K) \geq 2L/K$ for all $l \in \mathcal{S}_1$, we use Lemma \ref{lem:bound_nl} and find that
\begin{equation*}
    n_l(T) \leq \frac{3\log(T/\delta)K^2}{4L^2}.
\end{equation*}
Thus, the number of arms pulled in $\mathcal{S}_1$ is at most $3\log(T/\delta)K^2/(4L^2) \times 32LQ = 24 \log(T/\delta)K^2Q/L$. Since each arm in $\mathcal{S}_1$ has a payment larger than $M-32L/K$, the arms pulled from intervals in $\mathcal{S}_1$ contributes to $R_{subopt}$ by at most $24 \log(T/\delta)K^2Q/L\times 32L/K \leq 768\log(T/\delta)KQ$.

Finally, note that for $n \geq 2$ and $l\in \mathcal{S}_n$, $\Delta_{k,l} \geq (2^{n+2}-8)L/K \geq 2^{n+1}L/K$. Using Lemma \ref{lem:bound_nl}, we find that
\begin{equation*}
    n_l(T) \leq \frac{3\log(T/\delta)K^2}{2^{2n+2}L^2}.
\end{equation*}
Applying Lemma \ref{lem:lipschitz}, we see that each arm $a_i \in I_l$ verifies $m(a_i) \geq M - 2^{n+4}L/K$. Using Assumption \ref{hyp:beta}, we find that $\vert \mathcal{S}_n\vert\leq 2^{n+4}QL$. Thus,
\begin{eqnarray*}
R_{subopt} &\leq& \frac{600L^2QN}{K^2} + 768\log(T/\delta)KQ \\
&&+ \overset{\log_2(K/L)-2}{\underset{n = 2}{\sum}} \left \vert \mathcal{S}_n \right \vert \frac{3\log(T/\delta)K^2}{2^{2n+2}L^2}\times \frac{2^{n+4}L}{K}\\
&\leq& \frac{600L^2QN}{K^2} + 768\log(T/\delta)KQ +  192(\log(T/\delta)QK(\log_2(K/L)-3))\\
&\leq& \frac{600L^2QN}{K^2} + 192\log(T/\delta)KQ\left(\log_2(K/L) + 1\right).
\end{eqnarray*}

\subsection{Proof of Lemma \ref{lem:n_subopt}}

Along the lines of the proof of Lemma \ref{lem:R_subopt}, we have proved that on $\mathcal{E}_a \cap \mathcal{E}_b$, the number of arms in $\mathcal{S}_0$ is bounded by $30QLN/K$ and that the number of arms pulled from intervals in $\mathcal{S}_1$ is bounded by $24 \log(T/\delta)K^2Q/L$. Thus,
\begin{eqnarray*}
\underset{n = 0}{\overset{\log_2(K/L)-1}{\sum}}\underset{I_k \in \mathcal{S}_n}{\sum} n_k(T) &\leq& \frac{30QLN}{K} + \frac{24 \log(T/\delta)K^2Q}{L} + \underset{n = 2}{\overset{\log_2(K/L)-1}{\sum}}\left\vert \mathcal{S}_n \right\vert \frac{3\log(T/\delta)K^2}{2^{2n+2}L^2}\\
&\leq& \frac{30QLN}{K} + \frac{24 \log(T/\delta)K^2Q}{L} +  \underset{n = 2}{\overset{\log_2(K/L)-1}{\sum}} \frac{12\log(T/\delta)QK^2}{2^{n}L}\\
&\leq& \frac{30QLN}{K} + \frac{30 \log(T/\delta)K^2Q}{L}
\end{eqnarray*}

Thus, the number of arms pulled from sub-optimal intervals is bounded by $30Q(LN/K + \log(T/\delta)K^2/L)$.


\subsection{Proof of Lemma \ref{lem:used_up}}
Before proving Lemma \ref{lem:used_up}, let us introduce further notations. For any two intervals $I_h$ and $I_i$ such that $m_h \geq m_i$, define $N_{[h,i]} = \overset{i}{\underset{j = h}{\sum}} N_j$, and $n_{[h,i]}(T) = \overset{i}{\underset{j = h}{\sum}} n_j(T)$. 

We prove Lemma \ref{lem:used_up} by contradiction. Assume that there is an interval $I_k$ such that $m_k \geq M + AL/K$ and $n_k(T) < N_k$. By continuity of $m$, there exists $a \in [0,1]$ such that $m(a) = M + AL/(4K)$, and by Lemma \ref{lem:lipschitz} there exists an interval $I_l$ that contains $a$ such that $m_l \in [M + AL/(8K), M + AL/(2K)]$. Note that since $A\geq 33$, on the event $\mathcal{E}_a$ $m_l > m_{\hat{f}}$ and  $l<\widehat{f}$.

By definition of $\hat{f}$, we have $T > N_{[1,\widehat{f}]} = N_{[1,l-1]} + N_{[l,\widehat{f}]}$. On the other hand, $T =  n_{[1,l-1]}(T) + n_{[l,K]}(T)$. Since $N_{[1,l-1]}> n_{[1,l-1]}(T)$ on the event $\{n_k(T) < N_k\}$, we necessarily have $N_{[l, \widehat{f}]} < n_{[l, K]}(T) = n_{[l, \hat{f}]}(T) + n_{[\hat{f}+1, K]}(T)$.

We obtain a contradiction by proving that on $\mathcal{E}_a \cap \mathcal{E}_b \cap \left\{n_k(T) < N_k\right\}$, $$N_{[l, \widehat{f}]} - n_{[l, \widehat{f}]}(T) > n_{[\hat{f}+1, K]}(T).$$

In words, we prove that the number of sub-optimal arms pulled is strictly smaller than the number of remaining optimal arms, and obtain a contradiction.

To obtain a lower bound on $N_{[l, \widehat{f}]} - n_{[l, \widehat{f}]}(T)$, we note that for all  $h \in [l, \hat{f}]$, $\Delta_{k,h} \geq \Delta_{k,l} \geq AL/(2K)$. Using Lemma \ref{lem:bound_nl}, we see that of the event $\mathcal{E}_{a}\cap \mathcal{E}_b \cap \left\{n_k(T) < N_k \right\}$
\begin{eqnarray*}\label{eq:upper_n(T)}
n_{h}(T) \leq \frac{3\log(T/\delta)}{\Delta_{k,h}^2} \leq \frac{3\log(T/\delta)}{\left( AL/(2K)\right)^2} \leq \frac{12K^2\log(T/\delta)}{\left( AL\right)^2}.
\end{eqnarray*}

On the event $\mathcal{E}_b$, each interval contains at least $N/(2K)$ arms. Thus, 

\begin{eqnarray*}\label{eq:upper_n(T)}
N_{h}-n_{h}(T) \geq \frac{N}{2K} -  \frac{12K^2\log(T/\delta)}{\left( AL\right)^2}.
\end{eqnarray*}

Let $\mathcal{N}_{[l, \widehat{f}]}$ denote the number of intervals $I_h$ for $h \in [l, \widehat{f}]$, and let $a^{(1)}\in I_l$ be such that $m(a^{(1)}) = M+AL/(4K)$. Let $a^{(2)} = \argmin_{a : m(a) = M+ 4L/K}\vert a - a^{(1)}\vert$, and assume without loss of generality that $a^{(1)}< a^{(2)}$. Let $a^{(3)} = max\{a \in \left[a^{(1)}, a^{(2)}\right] : m(a) = M +AL/(4K)\}$. All interval $h$ such that $I_h \subset [a^{(3)},a^{(2)}]$ have mean reward in $[M+4L/K, M+AL/(4K)]$. On the event $\mathcal{E}_a$, those intervals belong to $[l, \widehat{f}]$. Using Assumption \ref{hyp:lip}, we find that $L\vert a^{(2)} - a^{(3)}\vert\lor 4L/K \geq (A-16)L/(4K)$, and so $\vert a^{(2)} - a^{(3)}\vert \geq (A-16)/(4K) \geq A/(8K)$ (since $A> 32$). The number of intervals of size $1/K$ in $[a^{(3)},a^{(2)}]$ is therefore at least $A/8 - 1$. Thus $\mathcal{N}_{[l, \widehat{f}]} \geq A/8 -1 $, and
\begin{equation*}
N_{[l, \widehat{f}]} - n_{[l, \widehat{f}]}(T) \geq \left(\frac{A}{8}-1\right) \left(\frac{N}{2K} -  \frac{12K^2\log(T/\delta)}{\left( AL\right)^2}\right).
\end{equation*}
Since $A>32$, $A/8 - 1 \geq 3A/32$. Thus
\begin{equation*}\label{eq:n_opt}
N_{[l, \widehat{f}]} - n_{[l, \widehat{f}]}(T) \geq \frac{3A}{32} \left(\frac{N}{2K} -  \frac{12K^2\log(T/\delta)}{\left( AL\right)^2}\right).
\end{equation*}
To obtain an upper bound on $ n_{[\hat{f}+1, K]}(T)$, we divide the intervals $\hat{f}+1, ..., K$ into subsets. Let $\widetilde{\mathcal{S}}_0= \{ l : M - m_l \in [-4L/K, AL/K]\}$, and for $n>0$ let $\widetilde{\mathcal{S}}_n= \{ l : M - m_l \in [AL/K \times 2^{n-1}, AL/K \times 2^{n}]\}$. Since $m_{\hat{f}} \leq M + 4L/K $, we see that $\{\hat f+1, ..., K\} \subset \underset{n \geq 0}{\cup}\widetilde{\mathcal{S}}_n$.

For all $h \in \widetilde{\mathcal{S}}_0$, $\Delta_{k,h} \geq (A-4)L/K \geq 7AL/(8K)$ since $A>32$. Similarly,  for all $n>0$ and all $h \in \widetilde{\mathcal{S}}_n$, $\Delta_{k,h} \geq AL/(K(1 + 2^{n-1})) \geq AL/(2^{n-1}K)$. Using Lemma \ref{lem:bound_nl}, we find that on the event $\mathcal{E}_{a} \cap \mathcal{E}_{b} \cap\left\{n_k(T) < N_k \right\}$,

\begin{eqnarray*}
n_{[\hat{f}+1, K]}(T) &\leq& \vert\widetilde{\mathcal{S}}_0 \vert\frac{192K^2\log(T/\delta)}{49A^2L^2} +  \underset{n\geq 1}{\sum} \vert\widetilde{\mathcal{S}}_n \vert \frac{3\log(T/\delta)}{\left(AL/K\right)^22^{2n-2}}
\end{eqnarray*}

Using Lemma \ref{lem:lipschitz} and Assumption \ref{hyp:beta}, we find that $\vert\widetilde{\mathcal{S}}_0 \vert\leq 2ALQ$ and that for any $n>0$, $\vert\widetilde{\mathcal{S}}_n \vert \leq 2^{n+1}ALQ$. This implies that 
\begin{eqnarray*}
n_{[\hat{f}+1, K]}(T) &\leq& \frac{384QK^2\log(T/\delta)}{49AL} +  \underset{n\geq1}{\sum}  \frac{48\log(T/\delta) QK^2}{AL2^n}\\
&\leq& \frac{48K^2\log(T/\delta)Q}{AL} \left(\frac{8}{49} + \underset{n\geq1}{\sum} \frac{1}{2^n} \right)\\
&\leq& \frac{56K^2\log(T/\delta)Q}{AL}
\end{eqnarray*}
 
Note that we necessarily have $QL \geq 1$. Thus, for the choice $A = 35\sqrt{\frac{K^3Q\log(T/\delta)}{NL}\lor 1}$, we find that $N_{[l, \widehat{f}]} - n_{[l, \widehat{f}]}(T) > n_{[\hat{f}+1, K]}(T)$, which is impossible. We conclude that all intervals $I_h$ with a mean reward larger than $AL/K$ have been killed.

\subsection{Proof of Lemma \ref{lem:bound_Ropt}}

We have shown in Lemma \ref{lem:n_subopt} that the number of non-zero terms in $R_{opt}$ is bounded by $30Q(LN/K +\log(T/\delta)K^2/L)$. Moreover, in Lemma \ref{lem:used_up}, we have shown that those non-zero terms correspond to arms $a_i$ in intervals $I_k$ such that $m_k \leq M + AL/K$. By Assumption \ref{hyp:lip}, their payments $m(a_i)$ are such that $m(a_i) \leq M + 2AL/K$, so each non zero term is bounded by $2AL/K$. Thus, we find that
\begin{equation*}
    R_{opt}\leq 30Q\left(\frac{LN}{K}+ \frac{\log(T/\delta)K^2}{L}\right)\times 2AL/K
\end{equation*}

\subsection{Proof of Lemma \ref{lem:h2h3}}

The functions $m_0$ and $m_1$ are piecewise linear with slopes $\tilde{L}$ and $-\tilde{L}$. Since $\tilde{L} = L \land 1/2 \leq L$, Assumption \ref{hyp:lip} is satisfied.

On the other hand, for $\epsilon \in (0,\tilde{L}\delta)$,
\begin{eqnarray*}
\lambda\left( \left\{x : \vert m_0(x) - 0.5 \vert \leq \epsilon \right\}\right) &=& \lambda\left([x_0 - \epsilon/\tilde{L}, x_0 + \epsilon/\tilde{L}] \right) +\lambda\left( [1-p - \epsilon/\tilde{L}, 1-p + \epsilon/\tilde{L}] \right)\\
&&+ \lambda\left( [x_1 - \epsilon/\tilde{L}, x_1 + \epsilon/\tilde{L}] \right)\\
&=& 6\epsilon/\tilde{L} = 6\epsilon \times (1/L \lor 2) \leq Q\epsilon.
\end{eqnarray*}

For $\epsilon \geq \tilde{L}\delta$,
\begin{eqnarray*}
\lambda\left( \left\{x : \vert m_0(x) - 0.5 \vert \leq \epsilon \right\}\right) &=& \lambda\left([x_0 - \epsilon/\tilde{L}, x_1 + \epsilon/\tilde{L}]\right) =  x_1 - x_0 + 2\epsilon/\tilde{L}\\
&=& 4\delta + 2\epsilon/\tilde{L} \leq 6\epsilon \times (1/L \lor 2) \leq Q\epsilon.
\end{eqnarray*}

Thus, $m_0$ satisfies Assumption \ref{hyp:beta}. The same holds for $m_1$.

\subsection{Proof of Lemma \ref{lem:KL_01}}

Recall that $\Phi(T) = \{\phi(1), ..., \phi(T)\}$. We bound the Kullback-Leibler divergence between $\mathbb{P}_0$ and $\mathbb{P}_1$ (see, e.g., Lemma 15.1 in \cite{lattimore2018bandit}):

\begin{eqnarray*}
KL(\mathbb{P}_0, \mathbb{P}_1) &=& \underset{i = 1,...,N}{\sum}\mathbb{E}_0\left[\mathds{1}_{i \in \Phi(T)} \right] KL(\mathcal{P}_0^{y_i}, \mathcal{P}_1^{y_i})\\
&\leq & \underset{i = 1,...,N}{\sum} KL(\mathcal{P}_0^{y_i}, \mathcal{P}_1^{y_i})
\end{eqnarray*}
where  $KL(\mathcal{P}_0^{y_i}, \mathcal{P}_1^{y_i})$ denotes the Kullback-Leibler divergence of the distribution of the reward $y_i$ under $m_0$ and $m_1$. For $p, q \in (0,1)$, we denote by $kl(p,q)$ the Kullback-Leibler divergence between two Bernoulli of means $p$ and $q$. Since the variables $y_i$ are Bernoulli random variable of parameter $m(a_i)$, we find that
\begin{eqnarray*}
KL(\mathbb{P}_0, \mathbb{P}_1) &\leq & \underset{i = 1, ..., N}{\sum} kl(m_0(a_i), m_1(a_i))\\
&\leq & \underset{a_i \in [x_0, x_1]}{\sum} kl(m_0(a_i), m_1(a_i)).
\end{eqnarray*}

By definition of $m_0$ and $m_1$, for all $a_i \in [x_0, x_1]$, $\vert 0.5 -  m_0(a_i)\vert = \vert 0.5 -  m_1(a_i)\vert \leq \delta \tilde{L} < 1/4$. Easy calculations show that for $\epsilon \in [-1/2, 1/2]$,
$$kl\left(\frac{1-\epsilon}{2},\frac{1+\epsilon}{2}\right) \leq 4\epsilon^2.$$ 
Using Assumption \ref{hyp:a_lb} and the definition of $m_0$ and $m_1$, we find that

\begin{eqnarray*}
KL(\mathbb{P}_0, \mathbb{P}_1) &\leq & 4 \overset{\lceil N\delta \rceil}{\underset{i = 0}{\sum}} 4 \left(2\tilde{L}\frac{i}{N} \right)^2\\
&\leq& \frac{64\tilde{L}^2}{N^2} \times \left( N\delta+1\right)^3.
\end{eqnarray*}
Now, since $\alpha \geq 20N^{-2/3}$ and $\tilde{L}^{-2/3} \geq 0.5^{-2/3}$, $N\delta = N^{2/3}\alpha\tilde{L}^{-2/3}\geq 31$, and thus $\left( N\delta+1\right)^3 \leq \left(N\delta\right)^3(1+1/31)^3 \leq 1.1\left(N\delta\right)^3$. Thus,
\begin{eqnarray*}
KL(\mathbb{P}_0, \mathbb{P}_1) \leq \frac{70.4\tilde{L}^2}{N^2} \times \left( N\delta\right)^3 \leq 70.4 \alpha^3.
\end{eqnarray*}

\subsection{Proof of Lemma \ref{lem:R(m)}}

Under $\mathbb{P}_0$, we can see that all arms in $(x_0, 1-p)$ are sub-optimal. By construction, all optimal arms have a payment higher than $1/2$. Thus,
\begin{eqnarray*}
R_T &\geq& \underset{a_i \in (x_0, 1-p)}{\sum} \left(\frac{1}{2} - m_0(a_i)\right)\mathds{1}\{i \in \Phi(T)\}.
\end{eqnarray*}
There are at least $\lfloor N\delta \rfloor$ arms in $[x_0, x_0 + \delta)$, and at least $\lfloor N\delta \rfloor$ arms in $[x_0 + \delta, 1-p]$. We use the change of variables $k = i - \lceil x_0 N\rceil$ and $k = \lceil (1-p)N\rceil - i$ to sum over the indices of the sub-optimal arms. We find that
\begin{eqnarray*}
R_T &\geq& \overset{\lfloor N\delta \rfloor-1}{\underset{k = 0}{\sum}} \frac{k\tilde{L}}{N} \mathds{1}\left\{(\lceil x_0 N\rceil + k) \in \Phi(T)\right\}\\
&& + \overset{\lfloor N\delta \rfloor}{\underset{k = 0}{\sum}} \frac{k\tilde{L}}{N} \mathds{1}\left\{((1-p)N - k) \in \Phi(T)\right\}.
\end{eqnarray*}
On $\mathcal{Z}$, at least $\lfloor N\delta \rfloor - 2$ arms are pulled in $(x_0, 1-p)$, so easy calculations lead to 
\begin{eqnarray*}
R_T &\geq & 2 \overset{\lfloor 0.5 N\delta \rfloor-2}{\underset{k = 0}{\sum}} \frac{k\tilde{L}}{N}\\
&\geq & \frac{2\tilde{L}}{N}\frac{(\lfloor 0.5 N\delta \rfloor-1)(\lfloor 0.5 N\delta \rfloor - 2)}{2}.
\end{eqnarray*}
We have shown in Lemma \ref{lem:KL_01} that $N\delta \geq 31$, so $(\lfloor 0.5 N\delta \rfloor-1)(\lfloor 0.5 N\delta \rfloor - 2) \geq 2^{-2}(N\delta)^2(1 - 4/31)(1 - 6/31)$. Thus,
\begin{eqnarray*}
R_T &\geq & 2^{-2}(1 - 4/31)(1 - 8/31)\frac{(N\delta)^2\tilde{L}}{N} \geq 2^{-2}(1 - 4/31)(1 - 6/31)\tilde{L}^{-1/3}\alpha^2N^{1/3}.
\end{eqnarray*}

Since $\tilde{L} \leq 1/2$, this implies 

\begin{eqnarray*}
R_T &\geq & 0.22 \alpha^2N^{1/3}.
\end{eqnarray*}

On the other hand, all  arms in $(x_0, 1-p)$ are optimal for the payoff function $m_1$. Since all sub-optimal arms have a payment at most $1/2$, under $\mathcal{P}_1$,
\begin{eqnarray*}
R_T &\geq& \underset{a_i \in [x_0, 1-p]}{\sum} \left( m_1(a_i)- \frac{1}{2}\right)\mathds{1}\{i \notin \Phi(T)\}.
\end{eqnarray*}
Applying the argument developed previously, we find that 

\begin{eqnarray*}
R_T &\geq& \overset{\lfloor N\delta \rfloor -1 }{\underset{k = 0}{\sum}} \frac{k\tilde{L}}{N} \mathds{1}\left\{(k + \lceil x_0 N\rceil) \in \Phi(T)\right\}\\
&& + \overset{\lfloor N\delta\rfloor}{\underset{k = 0}{\sum}} \frac{k\tilde{L}}{N} \mathds{1}\left\{((1-p)N- k) \in \Phi(T)\right\}.
\end{eqnarray*}
On $\overline{\mathcal{Z}}$, at most $\lfloor N\delta \rfloor - 2 $ arms are pulled in $(x_0, 1-p)$. Under Assumption \ref{hyp:a_lb}, there are at least $\lfloor 2N\delta \rfloor - 2$ arms in $(x_0, 1-p)$. All of these arms are optimal for the payoff function $m_1$. Thus, on $\overline{\mathcal{Z}}$, the number of sub-optimal arms pulled is at least $\lfloor 2N\delta \rfloor - 2 - (\lfloor N\delta \rfloor -2)\geq \lfloor N\delta \rfloor$. Thus,
\begin{eqnarray*}
R_T &\geq & 2 \overset{\lfloor 0.5 N\delta \rfloor - 1}{\underset{k = 0}{\sum}} \frac{k\tilde{L}}{N} \geq \frac{\tilde{L}}{N}(0.5 N\delta- 2)(0.5 N\delta-1)\geq \frac{2^{-2}(N\delta)^2\tilde{L}}{N}(1-4/31)(1-2/31).
\end{eqnarray*}
We use $\tilde{L} \leq 1/2$ to find that $R_T \geq   0.25\alpha^2N^{1/3}$.

\subsection{Proof of Lemma \ref{lem:isoperimetrie}}
By definition of $M$,
\begin{eqnarray*}
\mathbb{P}\left(m(a_1) \in [M, M+t)\right) &=& \mathbb{P}\left(m(a_1) \geq M \right) - \mathbb{P}\left(m(a_1) \geq M+t\right)\\
&=& p - \mathbb{P}\left(m(a_1) \geq M+t\right).
\end{eqnarray*}

To provide an upper bound on $\mathbb{P}\left(m(a_1) \geq M+t\right)$, we use gaussian isoperimetric inequalities (see, e.g., Chapter 5.1 in \cite{vershynin_2018}). Those results can readily be extended to random variable uniformly distributed on the unit cube. To do so, we introduce a random normal variable  $z = (z_1, ..., z_d) \sim \mathcal{N}(0, I_d)$, and we denote by $F$ the c.d.f. of a $z_1$. Moreover, we introduce a new payment function $$\tilde{m} : (z_1, ..., z_d) \rightarrow m(F(z_1), ..., F(z_d)).$$ It is easy to see that $\tilde{m}(z)$ and $m(a_1)$ have the same distribution. Thus, by definition of $p$,
\begin{equation}\label{eq:p-Pbad}
\mathbb{P}\left(m(a_1) \in [M, M+t)\right) = p - \mathbb{P}\left(\tilde{m}(z) \geq M+t\right).
\end{equation}

Next, we show that $\tilde{m}$ verifies a weak Lipschitz Assumption. Indeed, for any $z = (z_1, ..., z_d) \in \mathbb{R}^d$, and $z' = (z_1', ..., z_d') \in \mathbb{R}^d$, by definition of $\tilde{m}$
\begin{eqnarray*}
\left\vert\tilde{m}(z) - \tilde{m}(z')\right \vert &=& \left\vert m\left(F(z_1), ..., F(z_d)\right) - m\left(F(z_1'), ..., F(z_d')\right) \right\vert\\
&\leq& \left\vert M - \tilde{m}(z)\right\vert\lor L\left\Vert \left(F(z_1), ..., F(z_d)\right) - \left(F(z_1'), ..., F(z_d')\right) \right\Vert_2
\end{eqnarray*}
where the last equation follows from Assumption \ref{hyp:lip}. Now, the gaussian c.d.f. $F$ is Lipschitz continuous, with Lipschitz constant equal to $(2\pi)^{-1/2}$. Thus,
\begin{eqnarray*}
\vert\tilde{m}(z) - \tilde{m}(z') \vert &\leq& \vert M - \tilde{m}(z)\vert \lor (L\times(2\pi)^{-1/2}\Vert z - z' \Vert_2)
\end{eqnarray*}

Thus, for all $z \in \mathbb{R}^d$ such that $\tilde{m}(z)\geq M + t$ and all $z' \in \mathbb{R}^d$ such that $\tilde{m}(z')< M$, necessarily  $\Vert z - z' \Vert_2 \geq \sqrt{2\pi}t/L$. 

Let us denote by $\mathcal{B}$ the set of Borel sets of $\mathbb{R}^d$, and by $d(z,A)$ the Euclidean distance between a point $z\in \mathbb{R}^d$ and a set $A\in \mathcal{B}$. Moreover, let us denote by $A = \{z \in \mathbb{R}^d : \tilde{m}(z) < M\}$ the sub-level set of level $M$ of the function $\tilde{m}$. By definition of $M$, we have $\mathbb{P}(A) \leq 1-p$.  Moreover, the results above show that $\{z \in \mathbb{R}^d : \tilde{m}(z) \geq M + t\} \subset \{z \in \mathbb{R}^d : d(z,A) \geq \sqrt{2\pi}t/L\}$. This implies that 
\begin{eqnarray}
\mathbb{P}\left(\tilde{m}(z) \geq M+t\right) &\leq& \mathbb{P}\left(d(z, A) \geq \sqrt{2\pi}t/L\right)\nonumber\\
&\leq& \underset{B \in \mathcal{B} : \mathbb{P}(B) \leq 1-p}{\sup} \mathbb{P}\left(d(z, B) \geq \sqrt{2\pi}t/L\right).\label{eq:sup-Pbad}
\end{eqnarray}

By Theorem 5.2.1 in \cite{vershynin_2018}, $\mathbb{P}\left(d(z, B) \geq \sqrt{2\pi} t/L\right)$ is maximized under the constraint $\mathbb{P}(B) \leq 1-p$ when $B$ is a half space of gaussian measure $1-p$. This is the case, for example, when $B = \{x \in \mathbb{R}^d : \langle x \vert e_1 \rangle \geq F^{-1}(p)\}$ and $e_1 = (1, 0, ..., 0)$ is the first vector of the canonical basis of $\mathcal{R}^d$. Then, 
\begin{equation*}
    \left\{z : d(Z, B) \geq \sqrt{2\pi}t/L\right\} = \left\{z = (z_1, ..., z_d) : z_1 \leq F^{-1}(p) - \sqrt{2\pi}t/L\right\}.
\end{equation*}
Then, Equation \eqref{eq:sup-Pbad} implies
\begin{eqnarray}
\mathbb{P}\left(\tilde{m}(z) \geq M+t\right) &\leq& P\left(z_1 \leq F^{-1}(p) - \sqrt{2\pi}t/L\right)\nonumber \\
&=& F\left(F^{-1}(p) - \sqrt{2\pi}t/L\right).\label{eq:sup-Pbad2}
\end{eqnarray}
Combining Equations \eqref{eq:p-Pbad} and \eqref{eq:sup-Pbad2}, we find that
\begin{eqnarray*}
\mathbb{P}\left(m(a_1) \in [M,M+t)\right) &\geq& p - F\left(F^{-1}(p) - \sqrt{2\pi}t/L\right)\\
&= & F(F^{-1}(p)) - F\left(F^{-1}(p) - \sqrt{2\pi}t/L\right).
\end{eqnarray*}
Using the c.d.f. of the normal distribution, we find that 
\begin{eqnarray}
\mathbb{P}\left(m(a_1) \in [M, M+t)\right)&\geq& \int_{F^{-1}(p) - \sqrt{2\pi}t/L}^{F^{-1}(p)}\frac{1}{\sqrt{2\pi}}e^{\frac{-z^2}{2}}dz\\
&\geq& \frac{t}{L}e^{\frac{-(F^{-1}(p) - \sqrt{2\pi}t/L)^2}{2}}.
\end{eqnarray}
We recall that $t/L \leq \sqrt{d}$, and conclude that
$$\mathbb{P}\left(m(a_1) \in [M, M+t)\right)\geq \frac{t}{L}e^{-(F^{-1}(p) - \sqrt{2\pi d})^2/2}.$$


\subsection{Proof of Lemma \ref{lem:d-control_hat_f}}
Recall that $\epsilon = \lceil K^{d-1}c_{p,d} \rceil$. Similarly to the one-dimensional case, we begin by proving that $\widehat{f} \geq f - \epsilon$. Since this inequality becomes trivial if $\epsilon \geq f$, we assume that  $\epsilon < f$. Recall that $f = \lfloor pK^d\rfloor$, and $\widehat{f}$ is such that $ N_{1} + .. +  N_{\widehat{f}}< T \leq N_{1} + .. +  N_{\widehat{f}+1}$. By definition, $N_{1}  + .. + N_{f-\epsilon} = \underset{1\leq i \leq N}{\sum} \mathds{1}_{\{a_i \in B_{1}\cup .. \cup B_{f-\epsilon}\}}$, where $\mathds{1}_{\{a_i \in B_{1}\cup .. \cup B_{f-\epsilon}\}}$ are independent Bernoulli random variables of parameter $\frac{f-\epsilon}{K^d}$. Using Hoeffding's inequality, we see that for all $t>0$,

\begin{eqnarray*}
\mathbb{P}\left( \underset{1\leq i \leq N}{\sum} \mathds{1}_{\{a_i \in B_{1}\cup .. \cup B_{f-\epsilon}\}} - \frac{(f-\epsilon)N}{K^d} \geq t \right) &\leq& \exp\left(-\frac{2t^2}{N}\right).
\end{eqnarray*}
Choosing $t= \epsilon N/K^d \geq c_{p,d} N/K$, we see that
\begin{eqnarray*}
\mathbb{P}\left( \underset{1\leq i \leq N}{\sum} \mathds{1}_{\{a_i \in B_{1}\cup .. \cup B_{f-\epsilon}\}} - \frac{(f-\epsilon)N}{K^d} \geq \frac{\epsilon N}{K^d} \right) &\leq& \exp\left(-\frac{2c_{p,d}^2N}{K^2}\right).
\end{eqnarray*}
Now, by definition, $f = \lfloor TK^d/N \rfloor$, and so $fN/K^d \leq T$. Thus, 
\begin{eqnarray}\label{eq:f-1}
\mathbb{P}\left( N_{1} + .. +N_{f-\epsilon}  \geq T \right) &\leq& \exp\left(-\frac{2c_{p,d}^2N}{K^2}\right).
\end{eqnarray}

This shows that with high probability, $N_{1} + .. +N_{f-\epsilon}<T$, which implies that $f-\epsilon < \hat{f}+1$. Similarly, we can show that with probability at least $1 - \exp\left(-\frac{2c_{p,d}^2N}{K^2}\right)$, $f + \epsilon + 1 \geq \widehat{f}$. Thus, with probability larger than $1- 2\exp\left(-\frac{2c_{p,d}^2N}{K^2}\right)$, $\vert f - \hat{f} \vert \leq 1+ \epsilon$.

\medskip
In a second time, we prove that $m_{f} \in [M-L\sqrt{d}/K, M+L\sqrt{d}/K]$. 

We first show that there are at least $\lceil pK^d \rceil$ bins $k$ such that $m_k \geq M - L\sqrt{d}/K$, or equivalently that there are at most $\lfloor (1-p)K^d \rfloor $ bins $k$ such that $m_k < M - L\sqrt{d}/K$. Indeed, for all $k$ such  that $m_k < M - L\sqrt{d}/K$, there exists $a \in B_k$ such that $m(a) < M-L\sqrt{d}/K$. Using Lemma \ref{lem:lipschitz}, we see that $\forall a \in B_k$, $m(a) \leq M$. By definition of $M$, there can be at most $\lfloor (1-p)K
^d\rfloor$ such bins. Therefore, there are at least $\lceil pK^d \rceil$ bins $k$ such that $m_k \geq M - L\sqrt{d}/K$. Since $f < \lfloor pK^d\rfloor$, this implies that $m_{f} \geq M - L\sqrt{d}/K$. Similar arguments show that $m_{f} \leq M + L\sqrt{d}/K$.

Now, recall that $\alpha = 4QL/c_{p,d} + 2/\sqrt{d}\times (1+3/K^{d-1})$. We show that $m_{f - \epsilon - 2} \leq M + \alpha L\sqrt{d}/K$. Note that by Assumption \ref{hyp:lip} and by definition of $M$, $\max_a \{m(a)\} \leq M + L\sqrt{d}$. Then, if $\alpha/K\geq 1$, $m_{f - \epsilon - 2} \leq M + \alpha L\sqrt{d}/K$ is automatically verified. We therefore restrict our attention to the case $\alpha/K<1$. Now, we show  that there are at least $\epsilon + 2$ bins $B_k$ such that $m_k \in [m_f,M+\alpha L\sqrt{d}/K]$. Applying Lemma \ref{lem:isoperimetrie} and Assumption \ref{hyp:beta}, we find that 
\begin{align*}
&\lambda\left(\left\{a : m(a) \in [M + 2L\sqrt{d}/K, M + \alpha L\sqrt{d}/(2K)] \right\}\right) \\
& = \lambda\left(\left\{a :  m(a) \in [M, M + \alpha L\sqrt{d}/(2K)] \right\}\right) - \lambda\left(\left\{a : m(a) \in [M, M + 2L\sqrt{d}/K] \right\}\right)\\
& \geq \alpha c_{p,d}\sqrt{d}/(2K) - 2QL\sqrt{d}/K = c_{p,d}(1 + 3/K^{d-1}).
\end{align*}
Using Lemma \ref{lem:d-lipschitz}, we see that all arms $a$ such that $ m(a) \in [M + 2L\sqrt{d}/K, M + \alpha L\sqrt{d}/(2K)] $ belongs to bins $B_k$ such that $m_k \in [M + L\sqrt{d}/K, M + \alpha L\sqrt{d}/K]$. Thus, the number of bins with mean reward in $[M + L\sqrt{d}/K, M + \alpha L\sqrt{d}/K]$ is at least $c_{p,d}\left(1+3/K^{d-1}\right) \times K^d$. By definition of $\epsilon$, this number is larger than $\epsilon + 2$. This proves that there are at least $\epsilon + 2$ bins $B_k$ such that $m_k \in [m_f,M+\alpha L\sqrt{d}/K]$, so $m_{f-\epsilon - 2} \leq M + \alpha L\sqrt{d}/K$. Therefore, $m_{\widehat{f}} \leq M + \alpha L\sqrt{d}/K$ and $m_{\widehat{f}+ 1} \leq M + \alpha L\sqrt{d}/K$ with probability larger than $1- 2\exp\left(-\frac{2c_{p,d}^2N}{K^2}\right)$.

Similarly, we can show that $m_{\widehat{f}} \geq M - \alpha L\sqrt{d}/K$ and $m_{\widehat{f} + 1} \geq M - \alpha L\sqrt{d}/K$ with probability larger than $1- 2\exp\left(-\frac{2c_{p,d}^2N}{K^2}\right)$.

\subsection{Proof of Lemma \ref{lem:d-E_a}}
Non-zero terms in $R_T^{(d)}$ correspond to pairs of arms $(i, j)$ such that $i$ is pulled by $\phi^d$ but not by $\phi^*$, and $j$ is pulled by $\phi^*$ but not by $\phi^d$. If an arm $i$ is pulled by $\phi^d$, it belongs to a bin $k$ such that $m_k \geq m_{\widehat{f}+1}$. On the event $\mathcal{E}_a$, $m_{\hat{f}+1} \geq M - \alpha \sqrt{d}L/K$. Using Lemma \ref{lem:d-lipschitz}, we find that $$m(a_i) \geq M - 2\alpha \sqrt{d}L/K.$$ On the other hand, if $i$ is not pulled by $\phi^*$, it must be such that $m(a_i) \leq \widehat{M}$. On the event $\mathcal{E}_a$, this implies that $m(a_i) \leq M + \sqrt{d}L/K$. Since there are at most $\frac{4\alpha \sqrt{d}LQ N}{K}$ arms in $[M-2\alpha\sqrt{d}L/K, M + \sqrt{d}L/K]$ on the event $\mathcal{E}_a$, there are at most $\frac{4\alpha\sqrt{d}LQ N}{K}$ arms that are selected by $\phi^d$ and not by $\phi^*$, and thus at most $\frac{4\alpha\sqrt{d}LQ N}{K}$ non-zero terms in $R_T^{(d)}$. 

Similarly to the one-dimensional case, the cost of pulling an arm $i$ selected by $\phi^d$ but not by $\phi^*$, instead of an arm $j$ selected by $\phi^*$ but not by $\phi^d$, is bounded by $2\vert\widehat{M}- m_{\widehat{f}+1}\vert \lor 2\sqrt{d}L/K \leq 2\alpha\sqrt{d}L/K$. To conclude, on the event $\mathcal{E}_a$ there are at most $\frac{4\alpha\sqrt{d}LQN}{K}$ non-zero terms in $R_T^{(d)}$, and each of them is bounded by $2\alpha\sqrt{d}L/K$. Thus, 
$$R_T^{(d)} \leq \frac{8\alpha^2dQL^2N}{K^2}.$$

\subsection{Proof of Lemma \ref{lem:d-equi_nk}}

Note that for $k \in \{1,...,K\}$, $B_k$ is a Borel set of measure $1/K^d$. Applying Bernstein's inequality, we find that for all $t>0$,
\begin{eqnarray*}
\mathbb{P}\left(\left \vert N_k - \frac{N}{K^d}\right \vert \geq t \right) &\leq& 2e^{\frac{-t^2}{2K^{-d}N + 2t/3}}.
\end{eqnarray*}
Choosing $t=K^{-d}N/2$, we find that
\begin{eqnarray*}
\mathbb{P}\left(\left \vert N_k - \frac{N}{K^d}\right \vert \geq \frac{N}{2K^d} \right) &\leq& 2e^{-\frac{N}{10K^{d}}}.
\end{eqnarray*}
A union bound for $k = 1, ..., K^d$ yields the result.

\subsection{Proof of Lemma \ref{lem:d-R_subopt}}

By Lemma \ref{lem:d-control_hat_f}, on the event $\mathcal{E}_a$, $m_{\hat{f}+1} \in [M-\alpha\sqrt{d}L/K, M+\alpha\sqrt{d}L/K]$. We group bins with mean rewards lower than $m_{\widehat{f}+1}$ into the following subsets.

Let $\mathcal{S}_0 = \left\{ k : (M - m_{k}) \in [-\alpha\sqrt{d}L/K, 2\alpha\sqrt{d}L/K] \right\}$, and for $n\geq 1$ define \\$\mathcal{S}_n = \left\{ k : (M - m_{k}) \in [2^{n}\alpha\sqrt{d}L/K, 2^{n+1}\alpha\sqrt{d}L/K] \right\}$. Note that for $n \geq \log_2(K/(\alpha\sqrt{d}L))$, $\mathcal{S}_n$ is empty since $m$ is bounded by $1$. 

Using Lemma \ref{lem:d-lipschitz}, we note that for all $l \in \mathcal{S}_0 $ and all $a \in B_l$, $\vert m(a) - M \vert \leq 4\alpha\sqrt{d}L/K.$ Using Assumption \ref{hyp:beta}, we conclude that $\left \vert \mathcal{S}_0 \right \vert \leq 4\alpha\sqrt{d}LQK^{d-1}$. On $\mathcal{E}_a$, there are at most $1.5N/K^d$ arms in each bin, so the number of arms in bins in $\mathcal{S}_0$ is at most $6\alpha\sqrt{d}LQN/K$. Moreover for all $l \in \mathcal{S}_0$ and all $a_i \in B_l$, $(M - m(a_i))\leq 4\alpha\sqrt{d}L/K$. Thus, the arms pulled from bins in $\mathcal{S}_0$ contributes to $R_{subopt}$ by at most $24\alpha^2dQL^2N/K^2$.

Similarly, for all $n \geq 1$, all $l \in \mathcal{S}_n$ and all $a \in B_l$, $\vert m(a) - M \vert \leq 2^{n+1}\alpha\sqrt{d}L/K.$ Using Assumption \ref{hyp:beta}, we conclude that $\left \vert \mathcal{S}_n \right \vert \leq 2^{n+1}\alpha\sqrt{d}LQK^{d-1}$. Moreover, by definition of $\widehat{f}$, there exists a bin $B_k$ with $m_k \geq m_{\widehat{f}+1}$ such that $n_k(T) < N_k$. Since $\Delta_{k,l}  \geq 2^{n}\alpha\sqrt{d}L/K- \alpha\sqrt{d}L/K \geq 2
^{n-1}\alpha\sqrt{d}L/K$ for all $l \in \mathcal{S}_n$, we use Lemma \ref{lem:d-bound_nl} and find that
\begin{equation*}
    n_l(T) \leq \frac{3\log(T/\delta)K^2}{2^{2n-2}\alpha^2L^2d}.
\end{equation*}
Thus,
\begin{eqnarray*}
R_{subopt} &\leq& \frac{24\alpha^2dL^2QN}{K^2} + \overset{\log_2(K/\alpha\sqrt{d}L)}{\underset{n = 1}{\sum}} 2^{n+2}\alpha\sqrt{d}QLK^{d-1} \times \frac{3\log(T/\delta)K^2}{2^{2n-2}\alpha^2L^2d}\times \frac{2^{n+1}\alpha\sqrt{d}L}{K}\\
&\leq& \frac{24\alpha^2dL^2QN}{K^2} + 96QK^d\log(T/\delta)\log_2(K/\alpha\sqrt{d}L).
\end{eqnarray*}

\subsection{Proof of Lemma \ref{lem:d-n_subopt}}

Using the notations and results established along the proof of Lemma \ref{lem:d-R_subopt}, we find that

\begin{eqnarray*}
\underset{n = 0}{\overset{\log_2(K/\alpha\sqrt{d}L)}{\sum}}\underset{B_k \in \mathcal{S}_n}{\sum} n_k(T) &\leq& 6\alpha\sqrt{d}LQN/K + \underset{n = 1}{\overset{\log_2(K/\alpha\sqrt{d}L)}{\sum}}2^{n+2}\alpha\sqrt{d}QLK^{d-1} \times \frac{3\log(T/\delta)K^2}{2^{2n-2}\alpha^2L^2d}\\
&\leq& 6\alpha\sqrt{d}LQN/K + \frac{24 \log(T/\delta)K^{d+1}Q}{\alpha\sqrt{d}L}.
\end{eqnarray*}


\subsection{Proof of Lemma \ref{lem:d-used_up}}
As in the one-dimensional case, we use the following notations : for any two bins $B_h$ and $B_l$ such that $m_h \geq m_l$, define $N_{[h,l]} = \overset{l}{\underset{k = h}{\sum}} N_k$, and $n_{[h,l]}(T) = \overset{l}{\underset{k = h}{\sum}} n_k(T)$. We prove Lemma \ref{lem:d-used_up} by contradiction. We assume that there exists a bin $B_k$ such that $m_k \geq M + A\sqrt{d}L/K$ and $n_k(T) < N_k$ and define $h$ such that $m_h \in \argmax_{l}\{m_l : m_l \leq M + A\sqrt{d}L/(2K)\}$. Then, the arguments used to prove Lemma \ref{lem:used_up} show that we necessarily have $N_{[h, \widehat{f}]} < n_{[h, \hat{f}]}(T) + n_{[\hat{f}+1, K^d]}(T)$. We obtain a contradiction by proving that on $\mathcal{E}_a \cap \mathcal{E}_b \cap \left\{n_k(T) < N_k\right\}$, $$N_{[h, \widehat{f}]} - n_{[h, \widehat{f}]}(T) > n_{[\hat{f}+1, K^d]}(T).$$

To obtain a lower bound on $N_{[h, \widehat{f}]} - n_{[h, \widehat{f}]}(T)$, we note that for all  $l \in [h, \hat{f}]$, $\Delta_{k,l} \geq \Delta_{k,h} \geq A\sqrt{d}L/(2K)$. Using Lemma \ref{lem:bound_nl}, we see that of the event $\mathcal{E}_{a}\cap \mathcal{E}_b \cap \left\{n_k(T) < N_k \right\}$
\begin{eqnarray*}\label{eq:upper_n(T)}
n_{l}(T) \leq \frac{3\log(T/\delta)}{\Delta_{k,l}^2} \leq \frac{3\log(T/\delta)}{\left( A\sqrt{d}L/(2K)\right)^2} \leq \frac{12K^2\log(T/\delta)}{\left( A\sqrt{d}L\right)^2}.
\end{eqnarray*}

On the event $\mathcal{E}_b$, each bin contains at least $N/2K^d$ arms. Thus, 

\begin{eqnarray*}\label{eq:upper_n(T)}
N_{h}-n_{h}(T) \geq \frac{N}{2K^d} -  \frac{12K^2\log(T/\delta)}{\left( A\sqrt{d}L\right)^2}.
\end{eqnarray*}

The following reasoning helps us obtain a lower bound on the number of bins $B_l$ for $l \in [h, \widehat{f}]$, denoted by $\mathcal{N}_{[h, \widehat{f}]}$. First, recall that on $\mathcal{E}_a$, $m_{\widehat{f}} \leq M + \alpha \sqrt{d}L/K$. Now, any arm $a$ such that $m(a) \in [M+2\alpha\sqrt{d}L/K, M + A\sqrt{d}L/4K]$ belongs to a bin $B_l$ such that $m_l \in [M + \alpha\sqrt{d}L/K,  M + A\sqrt{d}L/2K]$. By definition of $h$, this bin $B_l$ is such that $l \in [h, \widehat{f}]$.

Next, we use Lemma \ref{lem:isoperimetrie} to lower bound $\lambda\left(\{a : m(a) \in [M+2\alpha\sqrt{d}L/K, M + A\sqrt{d}L/4K]\}\right)$. We have assumed that there exists a bin with mean reward larger than $M + A\sqrt{d}L/K$, so we necessarily have $A/K \leq 1$. Using Assumption \ref{hyp:beta} and Lemma \ref{lem:isoperimetrie}, we find that for the constant $c_{p,d}$ appearing in Lemma \ref{lem:isoperimetrie}
\begin{eqnarray*}
&& \lambda \left( \{ a: m(a) \in [M+2\alpha \sqrt{d}L/K, M + A \sqrt{d}L/4K] \} \right) \\
&& = \lambda \left(\{a : m(a) \in [M, M + A\sqrt{d}L/4K]\} \right) - \lambda \left( \{a : m(a) \in [M,M+2\alpha\sqrt{d}L/K]\} \right) \\
&& \geq  c_{p,d} \frac{A\sqrt{d}}{4K} - 2\alpha\sqrt{d}QL/K.
\end{eqnarray*}

By definition of $A$, we have $A \geq 16\alpha QL/c_{p,d}$, and thus $c_{p,d} A\sqrt{d}/(4K) - 2\alpha\sqrt{d}QL/K \geq c_{p,d} A\sqrt{d}/(8K)$. Now, all arms in $\{a: m(a) \in [M+2\alpha\sqrt{d}L/K, M + A\sqrt{d}L/4K]\}$ belongs to bins in $[h, \widehat{f}]$. Since each of those bins have volume $K^{-d}$, we find that $\mathcal{N}_{[h, \widehat{f}]} \geq c_{p,d} \frac{A\sqrt{d}}{8}K^{d-1}$, and

\begin{equation*}
N_{[h, \widehat{f}]} - n_{[h, \widehat{f}]}(T) \geq \frac{Ac_{p,d}\sqrt{d}K^{d-1}}{8}\left(\frac{N}{2K^d} -  \frac{12K^2\log(T/\delta)}{\left( A\sqrt{d}L\right)^2}\right).
\end{equation*}

To obtain an upper bound on $ n_{[\hat{f}+1, K]}(T)$, we divide the bins $\hat{f}+1, ..., K$ into subsets. Let $\widetilde{\mathcal{S}}_0= \{ l : M - m_l \in [-\alpha\sqrt{d}L/K, A\sqrt{d}L/K\}$, and for $n>0$ let $\widetilde{\mathcal{S}}_n= \{ l : M - m_l \in [A\sqrt{d}L/K \times 2^{n-1}, A\sqrt{d}L/K \times 2^{n}\}$. Since $m_{\hat{f}} \leq M + \alpha\sqrt{d}L/K $, we see that $\{\hat f+1, ..., K\} \subset \underset{n \geq 0}{\cup}\widetilde{\mathcal{S}}_n$.

For all $l \in \widetilde{\mathcal{S}}_0$, $\Delta_{k,l} \geq (A-\alpha)\sqrt{d}L/K \geq 15A\sqrt{d}L/(16K)$ since $A>16\alpha$. Using Lemma \ref{lem:d-lipschitz} and Assumption \ref{hyp:beta}, we find that $\vert\widetilde{\mathcal{S}}_0 \vert\leq 2A\sqrt{d}LQK^{d-1}$. Similarly,  for all $n>0$ and all $l \in \widetilde{\mathcal{S}}_n$, $\Delta_{k,l} \geq A\sqrt{d}L(1 + 2^{n-1})/K \geq A\sqrt{d}L2^{n-1}/K$, and $\vert\widetilde{\mathcal{S}}_n \vert \leq A\sqrt{d}QL2^{n+1}K^{d-1}$. Using Lemma \ref{lem:d-bound_nl}, we find that on the event $\mathcal{E}_{a} \cap \mathcal{E}_{b} \cap\left\{n_k(T) < N_k \right\}$,

\begin{eqnarray*}
n_{[\hat{f}+1, K]}(T) &\leq& \frac{768K^{d+1}Q\log(T/\delta)}{225\sqrt{d}AL} +  \underset{n\geq 1}{\sum} AQ\sqrt{d}L2^{n+1}K^{d-1} \frac{3K^2\log(T/\delta)}{A^2dL2^{2n-2}}\\
&\leq& \frac{28K^{d+1}Q\log(T/\delta)}{AL\sqrt{d}}
\end{eqnarray*}

Recall that we necessarily have $QL \geq 1$, and that $c_{p,d} \leq 1$. Thus, for the choice 
$$A =  \sqrt{\frac{472QK^{d+2}\log(T/\delta)}{Nc_{p,d}Ld}} \lor 16\alpha QL/c_{p,d},$$
we find that $N_{[h, \widehat{f}]} - n_{[h, \widehat{f}]}(T) > n_{[\hat{f}+1, K]}(T)$, which is impossible. We conclude that all bins $B_l$ with a mean reward larger than $M+A\sqrt{d}L/K$ have been emptied.

\end{document}